\documentclass[a4paper,11pt]{article}
\usepackage[margin=2.2cm]{geometry}
\usepackage{tikz}
\usetikzlibrary{positioning, arrows.meta}

\usepackage{lineno}  
\usepackage{bm}
\usepackage{amsmath}
\usepackage{amsfonts}
\usepackage{color}
\usepackage{amssymb}
\usepackage{graphicx}
\usepackage{enumerate}
\usepackage{amsthm,amscd}
\usepackage[all]{xy}
\usepackage{authblk}

\allowdisplaybreaks

\usepackage{hyperref}
\usepackage[nameinlink,capitalise]{cleveref}

\hypersetup{colorlinks=true,linkcolor=red,citecolor=blue,urlcolor=blue} 

\newtheorem{theorem}{Theorem}[section]
\newtheorem{construction}[theorem]{Construction}
\newtheorem{corollary}[theorem]{Corollary}
\newtheorem{definition}[theorem]{Definition}
\newtheorem{conjecture}[theorem]{Conjecture}

\newtheorem{lemma}[theorem]{Lemma}

\newtheorem{proposition}[theorem]{Proposition}

\newtheorem{claim}[theorem]{Claim}

\newtheorem{obser}[theorem]{Observation}
\newtheorem{assumption}[theorem]{Assumption}

\newcommand{\m}{\mathcal}

\begin{document}


\title{A sharp product bound for non-trivial cross-intersecting families}

%

\author{
Yang Huang\thanks{Moscow Institute of Physics and Technology; 
E-mail: \url{1060393815@qq.com}.}
}
\date{}
\maketitle
\vspace{-0.5cm}
\begin{abstract}
Two families $\mathcal{A}, \mathcal{B} \subset \binom{[n]}{k}$ are cross-intersecting if $A \cap B \ne \emptyset$ for all
$A \in \mathcal{A}$ and $B \in \mathcal{B}$, and non-trivial
if neither $\m A$ nor $\m B$ is a star. Pyber proved that any two cross-intersecting families 
$\mathcal{A}, \mathcal{B} \subset \binom{[n]}{k}$ satisfy
$|\mathcal{A}||\mathcal{B}| \le \binom{n-1}{k-1}^2$, and the maximum is 
attained by two full stars. 
Frankl, as well as Frankl and Wang, conjectured that the sharp bound, when both families are required to be non-trivial, is $h(n,k)^2$, where $h(n,k) = \binom{n-1}{k-1} - \binom{n-k-1}{k-1} + 1$, the size of
the Hilton--Milner family. The cases $k=3$, and the range $k\ge8$ and $n\ge 4k$, were established earlier
by Frankl and by Frankl and Wang, respectively.

In this paper, we prove their conjecture in the full range. We show that every non-trivial cross-intersecting pair $\mathcal{A}, \mathcal{B} \subset \binom{[n]}{k}$ with $n \ge 2k$ and $k \ge 3$ satisfies $|\mathcal{A}||\mathcal{B}| \le h(n,k)^2$. Moreover, we characterize all extremal pairs. Whereas the corresponding sum problem admits asymmetric and
unbalanced extremizers, the product extremum forces a balanced,
symmetric-or-dual structure: the two families are isomorphic  when
$n>2k$ and complement-dual when $n=2k$.

Independently and contemporaneously with the present work, Frankl and
Wang obtained the same bound for $k\ge8$ and $n\ge2k+1$ by a different
method. Our proof combines a diversity technique with several new properties
of an extended shift operation.
Moreover, we show that the problem behaves differently for different uniformities, exhibiting new extremal configurations. In particular, we disprove a related conjecture proposed by Frankl and Wang. 
\end{abstract}




\section{Introduction}
A $k$-set is a $k$-element set. 
We use the standard notation $[n]$ for $\{1,2,\dots,n\}$ and  ${X\choose k}$ for the collection of all $k$-sets of $X$. 
Let $[a,b]:=\{a, a+1, \dots, b\}$, with the convention that $[a,b]=\emptyset$ if $a>b$.  
A {\it family} is a collection of sets, and it is called $k$-{\it uniform} if every member is a $k$-set. 
A {\it star} is a family in which all members contain a common (fixed) element, called its ${\it center}$, and
a {\it full star} is the family of all $k$-sets containing a fixed element. 
A family is said to be {\it non-trivial} if it is not a star. 
A family $\m F$ is {\it intersecting} if $A\cap B\ne \emptyset$ for all $A, B \in \m F$. 
Let us recall two classical results in extremal combinatorics. 
 
\begin{theorem}[Erd\H{o}s--Ko--Rado \cite{EKR1961}]
Let $n\ge 2k\ge 4$,  and $\mathcal{F}\subset {[n] \choose k}$ be an intersecting family. Then
$$  |\mathcal{F}| \le {n-1 \choose k-1}.$$
 Moreover, when $n>2k$, equality holds if and only if $\mathcal{F}$ is a full star.
\end{theorem}

Let us introduce the {\it Hilton--Milner family}, which is a non-trivial intersecting family.  
\[ \begin{matrix}
\m {HM}(n,k):=
\Bigl\{G\in {[n] \choose k}: 
1\in G, G\cap [2,k+1] \neq \emptyset \Bigr\}
\cup \Bigl\{ [2,k+1]  \Bigr\}.
\end{matrix} \]
Denote 
\[
h(n,k):=|\m {HM}(n,k)|={n-1 \choose k-1} -{n-k-1 \choose k-1} +1.
\]

\begin{theorem}[Hilton--Milner \cite{HM1967}]\label{hm-nontrivial}
Let $n\ge 2k\ge 6$, and $\mathcal{F}\subset {[n] \choose k}$ be 
an intersecting  family.
 If  $\mathcal{F}$  is non-trivial,  then
$$ |\mathcal{F}| \le {n-1 \choose k-1} -{n-k-1 \choose k-1} +1.$$
When $n>2k$, equality holds if and only if either $\mathcal F$ is isomorphic to $\mathcal{HM}(n,k)$, or $k=3$ and $\mathcal F$ is isomorphic to 
\[\m T(n):=\left\{F\in\binom{[n]}3:|F\cap[3]|\ge2\right\}.\]
\end{theorem}

These two theorems have inspired a vast body of work, 
including numerous reproofs, refinements, and extensions.
 To the best of the author's knowledge, Hilton and Milner \cite{HM1967} were the first to extend the intersection problem to the cross‑intersection setting. 
Two families $\m A, \m B$ are {\it cross-intersecting} if $A\cap B\ne \emptyset$ for all $A\in \m A, B \in \m B$. In addition, if neither $\m A$ nor $\m B$ is a star, then we say that $\m A, \m B$ are {\it non-trivial cross-intersecting}.

Pyber \cite{Pyber1986} showed that the maximum of 
$|\mathcal{A}||\mathcal{B}|$ over cross-intersecting families 
$\mathcal{A},\mathcal{B}\subset\binom{[n]}{k}$ 
is attained when both are full stars.
\begin{theorem}[Pyber \cite{Pyber1986}]\label{3-17-1}
Let  $ n \ge 2k \ge 4 $ and $ \mathcal{A}, \m B \subset \binom{[n]}{k} $ be cross-intersecting. Then
\[
|\mathcal{A}| |\mathcal{B}| \leq \binom{n-1}{k-1}^2.
\]
For $ n > 2k $, equality holds if and only if $ \mathcal{A} = \mathcal{B} = \left\{ S \in \binom{[n]}{k} : i \in S \right\} $ for some $i\in [n]$.
\end{theorem}
Frankl and Kupavskii \cite{3-13-2} gave a short proof of Theorem \ref{3-17-1}. 
Pyber \cite{Pyber1986} also showed that 
$|\mathcal{A}||\mathcal{B}|\le \binom{n-1}{a-1}\binom{n-1}{b-1}$ 
for all cross-intersecting families 
$\mathcal{A}\subset \binom{[n]}{a}$, 
$\mathcal{B}\subset \binom{[n]}{b}$ 
with $n\ge 2a+b-2$ and $a\ge b$.
Matsumoto and Tokushige \cite{MT1989} then proved that Pyber's result holds for all $n\ge 2a, a \ge b$.  

Frankl \cite{1-25-1} proposed the non‑trivial cross‑intersection problem, and established the best possible bound $|\m A|+|\m B|\le {n\choose k}-2{n-k\choose k}+{n-2k\choose k}+2$ for non-trivial cross-intersecting families $\m A, \m B \subset {[n]\choose k}$. 
Determining the maximum of $|\mathcal{A}| |\mathcal{B}|$ for non-trivial cross-intersecting families seems to be a harder task, and it had remained open in the full range $n\ge 2k \ge 6$. As a first step in this direction, Frankl and Wang \cite{11-6-1} proved that if $k\ge 8$, $ n\ge 4k$, and $\mathcal{A} , \mathcal{B} \subset {[n] \choose k}$ are non-trivial cross-intersecting, then $|\mathcal{A}| |\mathcal{B}| \le h(n,k)^2$. 
Furthermore, they remarked that `the most intriguing problem is whether this bound holds in the full range'. The same problem was independently raised by Frankl \cite{1-25-1}. More recently, Frankl \cite{11-6-2} confirmed this bound 
in the case $k=3$. 

In this paper, we completely solve this problem by determining the maximum of $|\mathcal{A}| |\mathcal{B}|$ for all $n\ge 2k\ge 6$, and establishing all the extremal pairs. Let us state our main results.  
\begin{theorem}\label{thm:main}
Let $k \ge 3$ and $n\ge 2k$. If $\mathcal{A}, \mathcal{B} \subset {[n]\choose k}$ are non-trivial cross-intersecting, then  
\[
|\mathcal{A}||\mathcal{B}|\le h(n, k)^2. 
\]
\end{theorem}

Theorem~\ref{thm:main} gives the maximum value of the
product. We now describe all pairs attaining it. The
extremal pairs are built from the following three types.

\begin{construction}\label{extremal}
We define the following non-trivial cross-intersecting
pairs $(\m A,\m B)$ in $\binom{[n]}{k}$.

\smallskip
\begin{itemize}
\item {\bf Cross-HM type.} Let $n>2k\ge 6$. For $x\in[n]$,
$A_0,B_0\in\binom{[n]\setminus\{x\}}{k}$ with
$A_0\cap B_0\ne\emptyset$, let
\begin{align*}
&\m H_x(A_0;B_0):=\{A\in \binom{[n]}{k}: x\in A, A\cap B_0\ne \emptyset\}\cup \{A_0\},\\
&\m H_x(B_0;A_0):=\{B\in \binom{[n]}{k}: x\in B, B\cap A_0\ne \emptyset\}\cup \{B_0\}.
\end{align*}
We say that $(\m A, \m B)$ is of cross-HM type if $(\m A, \m B)=(\m H_x(A_0;B_0),\m H_x(B_0;A_0))$ for some $x,A_0,B_0$.
\item {\bf Cross-triangle type.} Let $n>2k$ and $k=3$. For $x\in[n]$ and
$P,Q\in\binom{[n]\setminus\{x\}}{2}$ with $P\cap Q\ne\emptyset$, set
\begin{align*}
\m T_x(P;Q):=\{A\in \tbinom{[n]}{3}: x\in A,\ A\cap Q\ne\emptyset\}
\cup\{A\in \tbinom{[n]}{3}: x\notin A,\ P\subset A\}.
\end{align*}
We say that $(\m A,\m B)$ is of cross-triangle type if
$(\m A,\m B)=(\m T_x(P;Q),\m T_x(Q;P))$ for some $x,P,Q$.
\item {\bf Complement-dual type.} Let $n=2k$. Let $\m C\subset\binom{[2k]}{k}$ be an arbitrary non-trivial family
with $|\m C|=h(2k,k)=\tfrac12\binom{2k}{k}$, and let 
\[
\m C^{\perp}:=\binom{[2k]}{k}\setminus\{[2k]\setminus C: C\in \m C\}.
\]
Note that the fact $\m C$ is non-trivial forces $\m C^{\perp}$ to be non-trivial as well.
We say that $(\m A, \m B)$ is of complement-dual type if $(\m A, \m B)=(\m C,\m C^{\perp})$.
\end{itemize}
\end{construction}

\begin{theorem}\label{thm:equality}
Let $n\ge 2k\ge 6$, and let
$\m A,\m B\subset\binom{[n]}{k}$ be non-trivial
cross-intersecting families with $|\m A||\m B|=h(n,k)^2$.
Then $(\m A,\m B)$ is isomorphic to one of the pairs in
Construction~\ref{extremal}.
\end{theorem}

\begin{remark}
Each pair in Construction~\ref{extremal} attains
$h(n,k)^2$, so these are exactly the extremal pairs.
Taking $A_0=B_0$ in cross-HM type recovers
$\m A=\m B=\m{HM}(n,k)$, and taking $P=Q$ in cross-triangle type recovers $\m A=\m B=\m{T}(n)$. 
\end{remark}

Theorem \ref{thm:main}  provides a product extension of the classical Hilton--Milner theorem and provides a sharp nontrivial analogue of Pyber’s product theorem. 
While the sum extremum is typically attained by \emph{asymmetric}
pairs, the product extremum forces a balanced,
symmetric-or-dual structure: the two families are isomorphic  when
$n>2k$ and complement-dual when $n=2k$; the full list of extremal pairs is
given in Theorem~\ref{thm:equality}. This reflects a general
phenomenon that product objectives favor balanced configurations.

\paragraph{Note added.}
During the final preparation of this manuscript, the author became
aware of an independent preprint of Frankl and Wang
\cite{FW2026},  
which proves the same bound
for $k\ge8$ and $n\ge2k+1$ by a different method. The two works were
developed independently.

\paragraph{Organization.}
The rest of this paper is organized as follows. 
Section~\ref{sec:preliminaries} introduces the notation and main tools used throughout the proof. 
In Section~\ref{sec:reduction}, we reduce Theorem~\ref{thm:main} to
Lemmas~\ref{mainlem0} -- \ref{mainlem2} and derive the product bound from
these three ingredients. 
Sections~\ref{sec:large-degree}--\ref{sec:asymmetric} prove 
Lemmas~\ref{mainlem0}--\ref{mainlem2}. 
The extremal pairs are characterized in Section~\ref{sec:equality}. Finally, 
Section~\ref{sec:concluding} discusses the problem for families of
different uniformities and disproves Conjecture~\ref{FW-conj}. The appendix contains the separate argument
needed for the equality case when $k=3$.

\section{Preliminaries}\label{sec:preliminaries}

{\bf Notation.}  Throughout this paper, we use the following standard notation.  
\begin{itemize}
\item For $i,j\in [n]$, $\m F \subset 2^{[n]}$, we denote $\m F(i):=\{F\setminus \{i\}: i\in F, F \in \m F\}$, $\m F(\bar{i}):=\{F\in \m F: i\not\in F \}$, $\m F(i\bar{j}):=\{F\setminus \{i\}: i\in F \in \m F, j\not\in F\}$, and $\m F(\bar{i}\bar{j}):=\{F\in \m F: i\not\in F, j\not\in F \}$. 
\item For a set $A\subset [n]$, write $\m F(A):=\{F\setminus A: A\subset F, F\in \m F\}$, and $\m F[A]:=\{F: A\subset F, F\in \m F\}$. When $A=\{i\}$, we  write $\m F[i]$ in place of $\m F[\{i\}]$.
\item $\m F(X\overline{Y}):=\{F\setminus X: X\subset F, F\cap Y=\emptyset, F\in \m F\}$.
\item Let
$\m A \subset {[n]\choose a}$
and
$\m B \subset {[n]\choose b}$
be non-trivial cross-intersecting families.
We say that
$\m A$ and $\m B$
are {\it maximal non-trivial cross-intersecting}
if, for any non-trivial cross-intersecting families
$\m A' \subset {[n]\choose a}$
and
$\m B' \subset {[n]\choose b}$
satisfying
$\m A \subset \m A'$
and
$\m B \subset \m B',$
we have
$ \m A'=\m A$ and 
$\m B'=\m B.$
\item A family $\mathcal F\subset 2^{[n]}$ is called {\it shifted} if 
$S_{i,j}(\mathcal F)=\mathcal F$ for all $1\le i<j\le n$. 
A pair of families $(\m A, \m B)$ is called {\it shifted} if both $\m A$ and $\m B$ are shifted. 
\item The {\it cover number} of a family $\m F$, denoted $\tau(\m F)$, is the minimum size of a set that intersects every set in $\m F$. 
\item For a family $\mathcal{F}$, we denote by $\Delta(\mathcal{F})$ the maximum degree of $\mathcal{F}$. The {\it diversity} of $\m F$ is defined as $\gamma(\m F) := |\m F|-\Delta(\m F).$
In other words, $\gamma (\m F)$ is the number of sets in $\m F$ not containing the element of maximum degree.
\item Let $x$ be an element of maximum degree in $\mathcal{F}$  (chosen arbitrarily if there are multiple). Denote by $\mathcal{F}_\Delta = \mathcal{F}(x)$ the {\it degree part}, and by $\mathcal{F}_\gamma = \mathcal{F}(\bar x)$ the {\it diversity part}.   
\item For two cross-intersecting families $\m F\subset {[n]\choose k}$ and $\m G \subset {[n]\choose \ell}$, we say that {\it $\m F$ is  maximal cross-intersecting with $\m G$} if there is no set $F\in {[n]\choose k}\setminus \m F$ such that $\m F \cup \{F\}$ and $\m G$ are still cross-intersecting.
\item For two cross-intersecting families $\m F\subset {[n]\choose k}$ and $\m G \subset {[n]\choose \ell}$, we say that 
{\it $\m F$ and $\m G$ are maximal cross-intersecting families}, if, for any pair of cross-intersecting families $\m F'\subset {[n]\choose k}$ and $\m G'\subset {[n]\choose \ell}$ with $\m F\subset \m F'$ and $\m G\subset \m G'$, we necessarily have $\m F= \m F'$ and $\m G= \m G'$.
\item For a family $\mathcal H\subset 2^{[n]}$ and $i\in[n]$, write $d_{\mathcal H}(i):=|\{H\in\mathcal H:i\in H\}|.$
\end{itemize}

We next recall the shifting operations used in the proof.
For $1\le i <j \le n$ and $F\subset [n]$, the classical $(i,j)$-shift, introduced by Erd\H{o}s, Ko, and Rado, is defined as follows: $S_{i,j}(F):=(F\setminus \{j\}) \cup \{i\}$ if $F\cap \{i,j\}=\{j\}$; and $S_{i,j}(F):=F$ otherwise. For a family $\m F$, define 
\[
S_{i,j}(\mathcal{F}) := \{ S_{i,j}(F) : F \in \mathcal{F} \} \cup \{ F \in \mathcal{F} : S_{i,j}(F) \in \mathcal{F} \}.
\]

\begin{lemma}[Frankl \cite{12-24-2}]\label{12-22-1}
Let $\mathcal{F}\subset 2^{[n]}$ be a family. If for some $1\le i<j \le n$, $S_{i,j}(\mathcal{F})$ is a star, then $F\cap \{i, j\}\ne \emptyset$ for all $F\in \mathcal{F}$, and $\mathcal{F}(i)\cap \mathcal{F}(j)=\emptyset$.
\end{lemma}

For a set $F$, we denote by $\min F$ and $\max F$ the smallest and largest element of $F$, respectively, 
with the convention that $\min \emptyset=\max \emptyset =0$.
We consider the lexicographic (lex for short) order on $2^{[n]}$ defined as follows. For two sets $A,B \subset [n]$, we write $A \prec B$ if $B\subset A$ or $\min (A\setminus B)<\min (B\setminus A)$ (note that $A\prec A$ for any $A\subset [n]$). 

For a set $X$ and integers $r,k$, we denote by $\m L(X,r,k)$ the family of the first $r$ $k$-subsets of $X$ in the lex  order.
We say that a family $\m F\subset {X\choose k}$ is {\it L-initial} if $\m F=\m L(X, |\m F|, k)$. 

Fix the increasing lex order on $\binom{[n]}{k}$, and let $i_{\text{lex}}(S)$ denote the position of a set $S$ in this ordering. For two families $\mathcal{F}_1, \mathcal{F}_2 \subset \binom{[n]}{k}$, we write 
$\mathcal{F}_1 \prec \mathcal{F}_2$ if 
$
\sum_{S \in \mathcal{F}_1} i_{\text{lex}}(S) \leq \sum_{S \in \mathcal{F}_2} i_{\text{lex}}(S),
$
with strict inequality indicating $\mathcal{F}_1 \precneqq \mathcal{F}_2$.
For pairs $(\mathcal{F}_1, \mathcal{G}_1)$ and  $(\mathcal{F}_2, \mathcal{G}_2)$, we write $(\mathcal{F}_1, \mathcal{G}_1) \precneqq  (\mathcal{F}_2, \mathcal{G}_2)$ to mean $\mathcal{F}_1 \prec \mathcal{F}_2$, $\mathcal{G}_1 \prec \mathcal{G}_2$, and at least one of $\mathcal{F}_1 \precneqq \mathcal{F}_2$ and  $\mathcal{G}_1 \precneqq \mathcal{G}_2$ holds.

Daykin \cite{Daykin1974} defined the following {\it $S_{U,V}$-shift}, which is a generalization of the classical shifting operation of Erd\H{o}s, Ko, and Rado \cite{EKR1961}. 
Let $\m F\subset 2^{[n]}$ and $F\in \m F$, and 
let $U,V\subset [n]$ be such that $|U|=|V|$ and $U\cap V=\emptyset$. 
We define 
\[ S_{U,V}(F):= 
\begin{cases}
    (F\setminus V) \cup U & \text{if $V\subset F$ and $U\cap F=\emptyset$,}\\
    F & \text{otherwise,}
\end{cases}\]

\[
S_{U,V}(\m F):=\{S_{U,V}(F): F\in \m F\}\cup \{F: S_{U,V}(F)\in \m F, F\in \m F\}.
\]
Throughout the paper, when using the notation $S_{U,V}$, we always assume that $U,V$ have the same size and are disjoint.
We now extend Lemma \ref{12-22-1} to the more general $S_{U,V}$-shift.
\begin{proposition}\label{1-21-2}
Let $\m F\subset {[n]\choose k}$ be a non-trivial family. If $S_{U,V}(\m F)$ is a star with center $i$, then $i\in U$ and $F\cap (U\cup V)=V$ for all $F\in \m F(\bar{i})$.
\end{proposition}

\begin{proof}
Since $\m F$ is not a star, $\m F(\bar{i})\ne \emptyset$.
Let $F\in \m F(\bar{i})$. 
As $S_{U,V}(\m F)$ is a star with center $i$, we have $S_{U,V}(F) \ne F$ and $i\in S_{U,V}(F)$,
 and hence
$i\in S_{U,V}(F)=(F\setminus V) \cup U$. 
Thus, $i\in U$ and $F\cap (U\cup V)=V$.
\end{proof}

Besides the classical shifts, we need a more flexible extended shift, which allows us to 
move several elements at once while preserving certain structure.
Let us introduce another tool--- $S_{U,V}^Q$-shift and the following `structure-preserving lemma', which is established by the author and Kupavskii \cite{HK2026}. For notational convenience, we define the following property $\bm{Q}$. 
\begin{itemize}
\item[{\bf Q}] Let $U,V\subset [n]$ satisfy 
$U\cap V=\emptyset$, $|U|=|V|$ and $U\prec V$.
We say that $(U,V)$ satisfies {\bf Q} for a family $\m F$ if $S_{U,V}(\m F)\precneqq \m F$, and 
for every $V'\subsetneqq V$, there exists $U'\subsetneqq U$ such that $U'\prec V'$, $|U'|=|V'|$, and $S_{U',V'}(\m F)=\m F$. 
\end{itemize}

To emphasize the choice of $(U,V)$ that satisfies {\bf Q}, we denote the corresponding $S_{U,V}$ by $S_{U,V}^Q$ (note that $S_{U,V}^Q$ and $S_{U,V}$ act identically).
For a family $\m F$, an $S_{U,V}^Q$-shift of $\m F$ means that $(U,V)$ satisfies {\bf Q} for $\m F$;
for a pair of families $(\mathcal A,\mathcal B)$,
an $S_{U,V}^Q$-shift of $(\mathcal A,\mathcal B)$
means that $(U,V)$ satisfies {\bf Q} 
for at least one of $\mathcal A$ and $\mathcal B$.
Throughout, an $S_{U,V}^Q$-shift of $\m F$ (or $(\mathcal A,\mathcal B)$) is understood to act non-trivially.
From the definition, if the $S_{i,j}$-shift acts on $\m F$ non-trivially, then it is also an $S_{\{i\},\{j\}}^Q$-shift of the family.

\begin{lemma}[Huang--Kupavskii \cite{HK2026}]\label{SUV'}
Let $n\ge a+b$, $\mathcal{A}\subset {[n]\choose a}$ and $\mathcal{B}\subset {[n]\choose b}$ be cross-intersecting. 
Let $(U,V)$ satisfy {\bf Q} for $(\m A, \m B)$. Then $S_{U,V}^{Q}(\mathcal{A})$ and $S_{U,V}^{Q}(\mathcal{B})$ are also cross-intersecting.
\end{lemma}

For  a set $C$ with $C\subset [n]$ and $|C|\le k$, denote
\[
\mathcal{E}^k_C:=\{C\cup T: T\in \tbinom{[\max C+1, n]}{ k-|C|}\}.
\]
The following definitions will be used frequently in the remainder of the paper.
\begin{itemize}
    \item We say that $C$ is \emph{full in $\mathcal{F}$} if every $k$-set $F$ with $C \subset F \subset [n]$ belongs to $\mathcal{F}$.
    \item We say that $C$ is \emph{empty in $\mathcal{F}$} if 
    $
    \mathcal{F} \cap \mathcal{E}^k_C = \emptyset.
    $
    \item We say that $C$ is {\it non-full in $\mathcal{F}$} if $\mathcal{E}^k_C \setminus \mathcal{F} \ne \emptyset $.

    Note that unlike the definition of `full', the definitions of `empty' and `non-full' depend on the relation between  $\mathcal{F}$ and the combinatorial block $\mathcal{E}^k_C$. 
\end{itemize}

Consider an $S_{U,V}^Q$-shift acting on $\mathcal{F}$, and let $\mathcal{S} \subset \mathcal{F}$ be a subfamily. 
We say that $\mathcal{S}$ is \emph{stable under the $S_{U,V}^Q$-shift} if an isomorphic copy of $\mathcal{S}$ is contained in $S_{U,V}^Q(\mathcal{F})$. When the shift is clear from context, we simply say that $\mathcal{S}$ is \emph{stable}.

Take a family $\m F\subset {X\choose k}$, let $R\subset X$ and $i\not\in R$ be such that $|R|<k$ and $r:=\min X\in R$. We define the following properties w.r.t. $\m F, R,r, i$: 

\begin{itemize}
\item[(P1)] For any $S_{U,V}$-shift of $\mathcal{F}$ with $r\in U$, if $S_{U,V}(\m F)\ne \m F$, then $|V|\ge |R|$;
\end{itemize}
\vspace{-0.3em}
\noindent
\begin{minipage}[t]{0.48\textwidth}
\begin{itemize}
\item[(P2)] $\exists F\in\mathcal{F}$ with $R\cap F=\emptyset$; 
\item[(P4)] $R\cup\{i\}$ is non-full in $\mathcal{F}$;
\end{itemize}
\end{minipage}
\hfill
\begin{minipage}[t]{0.48\textwidth}
\begin{itemize}
\item[(P3)] $R$ is non-full in $\mathcal{F}$;
\item[(P5)] $\exists F\in\m F$ with $R\cap F=\emptyset$ and $i\in F$.
\end{itemize}
\end{minipage}

\begin{lemma}[Structure-preserving lemma, Huang-Kupavskii \cite{HK2026}]\label{spl}
Suppose $\m F, R, i,r$ are defined as above.  
\begin{itemize}
\item[\rm (i)] If (P1)--(P3) hold, then  there exists an $S_{U,V}^Q$-shift of $\m F$ with $R\subset U$. 
\item[\rm (ii)] If (P1)--(P5) hold, then there exists an $S_{U,V}^Q$-shift of $\m F$ with  $R\subset U$ and $i\not\in V$.
\end{itemize}
\end{lemma}

Hilton \cite{Hilton1977} proved the following sharp bound for cross-intersecting families.

\begin{theorem}[Hilton \cite{Hilton1977}]\label{3-12-2}
Let  $ n \ge  2k \ge 4 $ and  $ \mathcal{A}, \mathcal{B} \subset \binom{[n]}{k} $ be cross-intersecting. Then  
\begin{equation}\label{3-17-2}
|\mathcal{A}| + |\mathcal{B}| \leq \binom{n}{k}.
\end{equation}
\end{theorem}

Frankl and Tokushige \cite{FT1992} generalized the Hilton--Milner theorem by giving the following result.

\begin{theorem}[Frankl--Tokushige \cite{FT1992}]\label{12-3-2}
If $\mathcal{A} \subset \binom{[n]}{a}$ and $\mathcal{B} \subset \binom{[n]}{b}$ are non-empty cross-intersecting families with $n \ge a + b$ and $a \le b$, then
\[
|\mathcal{A}| + |\mathcal{B}| \le \binom{n}{b} - \binom{n-a}{b} + 1.
\]
Moreover, for $n > a + b$, equality holds if and only if $\mathcal{A} = \{A\}$ and $\mathcal{B} = \{B \in \binom{[n]}{b} : B \cap A \neq \emptyset\}$, or for $a = b = 2$, there is one more possible family, $\mathcal{A} = \mathcal{B} = \{S \in \binom{[n]}{2} : x \in S\}$ for some $x \in [n]$.
\end{theorem}

\begin{theorem}[Hilton \cite{12-1-4}]\label{12-1-1}
Let $n\ge a+b$, and $\mathcal{A}\subset {[n]\choose a},\mathcal{B}\subset {[n]\choose b}$ be cross-intersecting. Then $\mathcal{L}([n], |\mathcal{A}|,a)$, $\mathcal{L}([n], |\mathcal{B}|,b)$ are cross-intersecting.
\end{theorem}

\begin{proposition}\label{lowerupper}
Let $n\ge a+b$, and let $\mathcal{A} \subset \binom{[n]}{a}$ and $\mathcal{B} \subset \binom{[n]}{b}$ be cross-intersecting families. For each $i\in [b]$, we have the following. 
\begin{itemize}
\item[(i)] If $|\m B|\ge {n-i \choose b-i}$,  then 
$
|\mathcal{A}|  \le {n\choose a}-{n-i\choose a}.
$
\item[(ii)] If $|\m B|\ge {n-1 \choose b-1}+{n-1-i\choose b-i}$,  then 
$
|\mathcal{A}|  \le {n-1\choose a-1}-{n-1-i\choose a-1}.
$
\end{itemize}
\end{proposition}

\begin{proof}
By Theorem \ref{12-1-1}, to prove Proposition \ref{lowerupper}  we may assume that $\m A, \m B$ are L-initial. 

Suppose $|\m B|\ge {n-i \choose b-i}$. Then all $b$-sets $B$ with $[i]\subset B\subset [n]$ belong to $\m B$. 
Given that $n\ge a+b\ge a+i$ and $1\le i\le b$, the cross-intersecting property of $\m A$ and $\m B$ implies that 
every set in $\m A$ must intersect $[i]$. 
Thus, $|\mathcal{A}|  \le {n\choose a}-{n-i\choose a}$. 
This proves (i).
 
Assume $|\m B|> {n-1 \choose b-1}$. 
Then every set in $\m A$ must intersect $1$. 
So $|\m A(1)|=|\m A|$. Moreover, as $\m A$ is L-initial, $\m A(1)$ is L-initial on ${[2, n]\choose a-1}$. 
Since $|\m B|\ge {n-1 \choose b-1}+{n-1-i\choose b-i}$ and $\m B$ is L-initial, $|\m B(\bar1)|\ge {n-1-i\choose b-i}$ and $\m B(\bar1)$ is L-initial on $[2,n]$. Note that $\m A(1)$ and $\m B(\bar1)$ are cross-intersecting and $n-1\ge a-1+b$. By (i), we obtain 
$|\m A|=|\m A(1)| \le {n-1\choose a-1}-{n-1-i\choose a-1}.
$ This proves (ii).
\end{proof}

\section{Proof of Theorem \ref{thm:main}}\label{sec:reduction}
We start with the following definition, which will be frequently used in this paper. 
\begin{definition}\label{lexmin}
Let $\m A,\m B\subset\binom{[n]}{k}$ be non-trivial cross-intersecting.
The pair $(\m A,\m B)$ is \emph{lex-minimal} if for every pair of
cross-intersecting families $\m C,\m D\subset\binom{[n]}{k}$ with
$|\m C|=|\m A|$, $|\m D|=|\m B|$ and $(\m C,\m D)\precneqq(\m A,\m B)$, at least
one of $\m C$ and $\m D$ is a star.
\end{definition}

We prove the product bound through the following three lemmas. The first lemma shows that, after choosing an extremal pair in a lex-minimal form, at least one of the two families has a large maximum degree. 

\begin{lemma}\label{mainlem0}
Let $n\ge 2k\ge 6$, and let $\mathcal{A}, \mathcal{B} \subset {[n]\choose k}$ be non-trivial  cross-intersecting families such that $|\m A||\m B|$ is maximum among all pairs of non-trivial  cross-intersecting families. 
Suppose further that $(\m A, \m B)$ is lex-minimal. 
If $|\mathcal{A}|\ge  h(n,k)$, then 
 $\Delta(\mathcal{A})\ge {n-2\choose k-2}+1$.
 \end{lemma}

The second lemma treats the case where both families have large maximum degree; in this regime, a sum-type estimate is strong enough to imply the desired product bound.

\begin{lemma}\label{mainlem1}
Let $n\ge 2k\ge 6$, and let $\mathcal{A}, \mathcal{B} \subset {[n]\choose k}$ be non-trivial cross-intersecting families such that $|\m A||\m B|$ is maximum among all pairs of non-trivial  cross-intersecting families.  
Suppose further that $(\m A, \m B)$ is lex-minimal. 
If $\Delta(\mathcal{A})\ge {n-2\choose k-2}+1$ and $\Delta(\mathcal{B})\ge {n-2\choose k-2}+1$, then  
$|\mathcal{A}|+|\mathcal{B}|\le 2h(n,k)$.
\end{lemma}

The third lemma treats the remaining degree-asymmetric case.

\begin{lemma}\label{mainlem2}
Let $n\ge 2k\ge 6$, and let $\mathcal{A}, \mathcal{B} \subset {[n]\choose k}$ be non-trivial cross-intersecting families such that $|\m A||\m B|$ is maximum among all pairs of non-trivial  cross-intersecting families.  
If $|\m A|\ge h(n,k)$, 
$\Delta(\m A)\ge {n-2\choose k-2}+1$ and $\Delta(\m B)\le {n-2\choose k-2}$, then $|\m A||\m B|\le h(n,k)^2$. 
\end{lemma}

We now assemble Lemmas \ref{mainlem0}, \ref{mainlem1}, and \ref{mainlem2} into a proof of the main theorem.

\begin{proof}[Proof of Theorem \ref{thm:main}]
Denote 
\[
\Pi:=\max\{\,|\m A||\m B| : \m A,\m B\subset \tbinom{[n]}{k}
\text{ are non-trivial cross-intersecting}\,\}.
\]
The Hilton--Milner family $\m{HM}(n,k)$ is non-trivial and intersecting,
so $\m{HM}(n,k)$ and $\m{HM}(n,k)$ are non-trivial cross-intersecting families with product of sizes $h(n,k)^2$. The pair above gives
\begin{equation}\label{lowerbd}
\Pi\ge h(n,k)^2.
\end{equation}
It remains to prove $\Pi\le h(n,k)^2$.
Recall the weight $w(\m F):=\sum_{S\in\m F}
i_{\mathrm{lex}}(S)$ and the relation $\precneqq$ on pairs. By definition,
$(\m C,\m D)\precneqq(\m A,\m B)$ entails $w(\m C)\le w(\m A)$,
$w(\m D)\le w(\m B)$ with at least one strict inequality, and hence
\begin{equation}\label{weightdrop}
 (\m C,\m D)\precneqq(\m A,\m B) \text{ implies }
w(\m C)+w(\m D)<w(\m A)+w(\m B).
\end{equation}
Among all non-trivial cross-intersecting pairs $(\m A,\m B)$ with
$|\m A||\m B|=\Pi$, fix one minimizing $w(\m A)+w(\m B)$. 
We claim that $(\m A,\m B)$ is lex-minimal. Indeed, let
$(\m C,\m D)$ be a cross-intersecting pair with
$|\m C|=|\m A|$, $|\m D|=|\m B|$ and
$(\m C,\m D)\precneqq(\m A,\m B)$. If both $\m C$ and $\m D$
were non-trivial, then $|\m C||\m D|=\Pi$ and, by~\eqref{weightdrop},
the total weight would be smaller, contradicting the choice of
$(\m A,\m B)$.

By~\eqref{lowerbd}, $|\m A||\m B|=\Pi\ge h(n,k)^2$, so
$\max\{|\m A|,|\m B|\}\ge h(n,k)$. 
By symmetry, we may assume $|\m A|\ge h(n,k).$
Lemma \ref{mainlem0} now yields $\Delta(\m A)\ge \binom{n-2}{k-2}+1.$
We split into two cases according to $\Delta(\m B)$.
If $\Delta(\m B)\ge \binom{n-2}{k-2}+1$, then $\Delta(\m A),\Delta(\m B)\ge \binom{n-2}{k-2}+1$, so
Lemma \ref{mainlem1} gives $|\m A|+|\m B|\le 2h(n,k)$, and hence
$\Pi=|\m A||\m B|\le\Big(\tfrac{|\m A|+|\m B|}{2}\Big)^2\le h(n,k)^2.$
If $\Delta(\m B)\le \binom{n-2}{k-2}$, then  
Lemma \ref{mainlem2} gives directly $\Pi=|\m A||\m B|\le h(n,k)^2.$ 
In either case $\Pi\le h(n,k)^2$. Together with \eqref{lowerbd}, this gives
$\Pi=h(n,k)^2$, and the value is attained by $\m A=\m B=\m{HM}(n,k)$. This
completes the proof of Theorem \ref{thm:main}.
\end{proof}

\section{Proof of Lemma \ref{mainlem0}}\label{sec:large-degree}
The purpose of this section is to prove Lemma \ref{mainlem0}. Starting from the
lex-minimal setting, we show that one of the two families must have
large maximum degree.
We begin with two observations. 

\begin{obser}\label{notlinitial}
Let $n\ge a+b$, and let $\m A\subset {[n]\choose a}$ and $\m B\subset {[n]\choose b}$ be non-trivial cross-intersecting families. Then $\m A$ and $\m B$ are not L-initial.
\end{obser}
\begin{proof}
Assume w.l.o.g. $\m A$ is L-initial.
Since $\m A$ is non-trivial, all $a$-sets $A$ with $1\in A\subset [n]$ belong to $\m A$.
Since $n\ge a+b$, the cross-intersecting property of $\m A$ and $\m B$ forces  every set in $\m B$ to contain $1$, contradicting the non-triviality of $\m B$.
\end{proof}

Recall that, by definition, an $S_{U,V}^Q$-shift of
$(\m A,\m B)$ acts non-trivially on $(\m A,\m B)$.
Moreover, it preserves the sizes of the two families, and, 
by Lemma \ref{SUV'}, it preserves the
cross-intersecting property. 
We therefore have the following observation.
\begin{obser}\label{12-25-2}
Let $n\ge 2k$ and $\mathcal{A}, \mathcal{B} \subset{[n]\choose k}$ be non-trivial cross-intersecting families. Suppose further that $(\m A, \m B)$ is lex-minimal (Definition \ref{lexmin}). Then, for any $S_{U,V}^Q$-shift of $(\m A, \m B)$, at least one of $S_{U,V}^Q(\mathcal{A})$ and $S_{U,V}^Q(\mathcal{B})$ is a star. 
\end{obser}
\begin{proposition}\label{1-1-2}
Let $n\ge 2k\ge 6$, and let $\mathcal{A}, \mathcal{B}\subset {[n]\choose k}$ be maximal non-trivial cross-intersecting families. Suppose that $S_{i,j}(\mathcal{B})$ is a star for some $1\le i<j \le n$. Then
all $k$-sets $A$ with $\{i,j\}\subset A \subset [n]$ belong to $\mathcal{A}$. Furthermore, $\mathcal{A}(i\bar{j})\ne \emptyset$, $\mathcal{A}(\bar{i}j)\ne \emptyset$ and $\Delta(\m A)\ge {n-2\choose k-2}+1$. 
\end{proposition}

\begin{proof}
Since $S_{i,j}(\mathcal{B})$ is a star, by Lemma \ref{12-22-1}, $B\cap \{i, j\}\ne \emptyset$ for all $B\in \mathcal{B}$, and $\mathcal{B}(i)\cap \mathcal{B}(j)=\emptyset$. 
By the  maximality of $(\mathcal{A}, \mathcal{B})$, all $k$-sets $A$ with $\{i,j\}\subset A \subset [n]$ belong to $\mathcal{A}$, and thus $|\m A(\{i,j\})|={n-2\choose k-2}$.
Since $\mathcal{B}$ is non-trivial, $\mathcal{B}(i\bar{j})\ne \emptyset$ and $\mathcal{B}(\bar{i}j)\ne \emptyset$. Take $B\in \mathcal{B}(i\bar{j})$, and write $B':=B\cup \{i\}$, $B'':=B\cup \{j\}$. Then $B'\in \m B$ and $B''\not\in \m B$ since 
$\mathcal{B}(i)\cap \mathcal{B}(j)=\emptyset$. 
The maximality of $(\mathcal{A}, \mathcal{B})$ implies that there exists $A\in \m A$ such that $A\cap B''=\emptyset$ and $A\cap B'\ne \emptyset$. 
This forces $i\in A$ and $j\not\in A$, so $A\in \mathcal{A}(i\bar{j})$, and hence $\mathcal{A}(i\bar{j})\ne \emptyset$. 
A symmetry argument gives $\mathcal{A}(\bar{i}j)\ne \emptyset$. 
Combining $|\m A(\{i,j\})|={n-2\choose k-2}$ with $\mathcal{A}(i\bar{j})\ne \emptyset$, we have $\Delta(\m A)\ge |\m A(i)|\ge {n-2\choose k-2}+1$. 
\end{proof}

Let $\m F\subset {X\choose k}$ and $i\in X$. We say that $\m F$ is {\it $i$-shifted} if for any $j>i$, $S_{i,j}$-shift acts trivially on $\m F$.

\begin{proposition}\label{suvcenter}
Let $i=\min X$ and $\m F\subset {X\choose k}$ be $i$-shifted.  If $S_{U,V}(\m F)$ is a star and $i\in U$, then $i$ is a center of $S_{U,V}(\m F)$.
\end{proposition}

\begin{proof}
Suppose for contradiction that $i$ is not a center of $S_{U,V}(\m F)$. 
Since $S_{U,V}(\m F)$ is a star, there is $j\in X$ such that $j$ is a center of $S_{U,V}(\m F)$. Then $j>i$ since $i=\min X$, and there exists $F\in S_{U,V}(\m F)$ such that $j\in F$ and $i\not\in F$. 
Let $F':=(F\setminus \{j\})\cup \{i\}$. Then $F'\not\in S_{U,V}(\m F)$. 
Since $i\in U$ and $i\not\in F$, $F\in S_{U,V}(\m F) \cap \m F$. As $\m F$ is $i$-shifted, $F'\in \m F$. 
Since $i\in F'\cap U$ and $V\cap U=\emptyset$, $S_{U,V}(F')=F'$. Hence $F'\in S_{U,V}(\m F)$, a contradiction.
\end{proof}

\subsection{Proof of Lemma \ref{mainlem0}}
\begin{proof}[Proof of Lemma \ref{mainlem0}]
Assume for contradiction that $\Delta(\mathcal{A})\le {n-2\choose k-2}$. 
As $|\m A(1)|\le \Delta(\m A)\le {n-2\choose k-2}$ and $|\m A|\ge h(n,k)$, 
we have 
\begin{equation}\label{abar1>}
    |\m A(\bar{1})|=|\m A|-|\m A(1)|\ge h(n,k)-{n-2\choose k-2}= {n-2\choose k-1}-{n-1-k\choose k-1}+1.
\end{equation}

\begin{claim}
    $\Delta(\m B)\le {n-2\choose k-2}$.
\end{claim}

\begin{proof}
We w.l.o.g. assume $\m B_{\Delta}=\m B(1)$. 
Observe that $\m A(\bar{1})\subset {[2,n]\choose k}$ and $\m B(1)\subset {[2,n]\choose k-1}$ are cross-intersecting. Since $n\ge 2k$, 
applying Proposition \ref{lowerupper}(ii) (with $[2,n], k, k-1, k-1$ in the roles of $[n], a, b, i$, respectively) together with~\eqref{abar1>}, we obtain 
$\Delta(\m B)=|\m B(1)|\le {n-2\choose k-2}.$ 
\end{proof}

\begin{claim}\label{abshifted}
    $\mathcal{A}$ and $\mathcal{B}$ are shifted.
\end{claim}

\begin{proof}
Assume for contradiction that there exist $1\le i <j \le n$ such that $(S_{i,j}(\mathcal{A}), S_{i,j}(\mathcal{B}))\precneqq(\mathcal{A}, \mathcal{B})$. 
Since $(\m A, \m B)$ is lex-minimal, Observation \ref{12-25-2} implies that $S_{i,j}(\mathcal{A})$ or $S_{i,j}(\mathcal{B})$ is a star. 
By the maximality of $|\m A||\m B|$, $\m A$ and $\m B$ are maximal non-trivial cross-intersecting. 
 Proposition \ref{1-1-2} gives  
$\Delta(\mathcal{A})\ge {n-2\choose k-2}+1$ or $\Delta(\m B)\ge {n-2\choose k-2}+1$, which is a contradiction to $\Delta(\mathcal{A})\le {n-2\choose k-2}$ and $\Delta(\m B)\le {n-2\choose k-2}$. Thus $\mathcal{A}$ and $\mathcal{B}$ are shifted.
\end{proof} 

\begin{claim}\label{12nonfull}
    \begin{itemize}
    \item[(i)] 
    $\m A(\bar1\bar2)\ne \emptyset$, $\m A(1\bar2)\ne \emptyset$, and $\m A(\bar1 2)\ne \emptyset$, and the same holds for $\m B$;
    \item[(ii)] 
    $\{1,2\}$ is non-full in both $\m A$ and $\m B$.
    \end{itemize}
\end{claim}

\begin{proof}
By Claim \ref{abshifted}, $\Delta(\mathcal{A})=|\mathcal{A}(1)|$, $\Delta(\mathcal{B})=|\mathcal{B}(1)|$ and $|\mathcal{A}(1)|\ge |\mathcal{A}(2)|\ge \dots \ge |\mathcal{A}(n)|$. 
Since $n\ge 2k$, ${n-2\choose k-1}-{n-1-k\choose k-1}+1>{n-2\choose k-2}$.
By~\eqref{abar1>}, we have  
$|\mathcal{A}(\bar{1})|>{n-2\choose k-2}$.
Since $\Delta(\m A)=|\mathcal{A}(1)|\le {n-2\choose k-2}$, $\mathcal{A}(\bar{1}\bar{2})\ne \emptyset$, otherwise $|\m A(2)|\ge |\m A(\bar12)|=|\m A(\bar1)|>{n-2\choose k-2}\ge |\mathcal{A}(1)|$, a contradiction. 
Since $\m A$ is shifted,  
$\mathcal{A}(1\bar{2})\ne \emptyset$ and $\mathcal{A}(\bar{1}2)\ne \emptyset$.  
For $\m B$, suppose that $\mathcal{B}(\bar{1}\bar{2})= \emptyset$, then by the maximality of $(\mathcal{A}, \mathcal{B})$, all $k$-sets $A$ containing $\{1, 2\}$ belong to $\mathcal{A}$, together with $\mathcal{A}(1\bar{2})\ne \emptyset$, this would imply $\Delta(\mathcal{A})>{n-2\choose k-2}$, a contradiction. Thus $\mathcal{B}(\bar{1}\bar{2})\ne \emptyset$. Since $\m B$ is shifted,  
$\mathcal{B}(1\bar{2})\ne \emptyset$ and $\mathcal{B}(\bar{1}2)\ne \emptyset$.      
This proves (i). Since $n\ge 2k$, (ii) follows by (i). 
\end{proof}

Note that $(\m A, \m B)$ is not L-initial (see Observation \ref{notlinitial}). 
We apply Lemma \ref{spl}(i) with $r=1$, $R=\{1,2\}$ and  $\m F=\m A$. 
The required conditions are verified as follows: Since $\m A$ is shifted, every $S_{U,V}$-shift of $\m A$
with $1\in U$ and $|V|=1$ acts trivially.
Thus any non-trivial $S_{U,V}$-shift with $1\in U$
satisfies $|V|\ge 2=|R|$, so (P1) holds. (P2) and (P3) hold by Claim \ref{12nonfull}. 
Hence there exists an $S^{Q}_{U,V}$-shift of $\m A$
(and hence of $(\m A,\m B)$) such that $\{1,2\}\subset U$.
By Observation \ref{12-25-2}, at least one of $S^{Q}_{U,V}(\m A)$ and $S^{Q}_{U,V}(\m B)$
is a star. 
By Proposition \ref{suvcenter}, the center of this star must be $1$.
However, since $2\in U$,
every set in $\mathcal{A}(\bar{1}2)$ and
$\mathcal{B}(\bar{1}2)$ remains unchanged under this $S_{U,V}^Q$-shift.
By Claim \ref{12nonfull},
neither $S^{Q}_{U,V}(\m A)$ nor
$S^{Q}_{U,V}(\m B)$ is a star, a contradiction.
This proves Lemma \ref{mainlem0}.
\end{proof}

\section{The large-degree case: Proof of Lemma \ref{mainlem1}}\label{sec:both-large}
In this section we prove Lemma \ref{mainlem1}, which handles the case where both families have large maximum degree. The goal is to replace the original families by suitable `local'- L-initial families without losing non-trivial cross-intersecting property. 
This reduction lets us estimate the sizes from the simpler structure.
We start with the following lower bound on $h(n,k)$.
 Since $n\ge 2k$,
\begin{equation}\label{hnk>}
    h(n,k)\ge {n-2\choose k-2}+2{n-3\choose k-2}.
\end{equation}

\begin{proposition}\label{1-20-1}
Let $n > k$, let $\m F\subset {[n]\choose k}$ be a non-trivial family such that $\m F_{\Delta}=\m F(1)$. Define the family $\m F'$ by  setting $\m F'(1)=\m L([2,n], |\m F(1)|, k-1)$ and $\m F'(\bar{1})=\m L([2,n], |\m F(\bar{1})|, k)$. 
If $\Delta(\m F)\ge {n-2\choose k-2}+\gamma(\m F)$, then $\m F'$ is non-trivial and $\m F'_{\Delta}=\m F'(1)$.
\end{proposition}

\begin{proof}
Since $\m F_{\Delta}=\m F(1)$ and $\Delta(\m F)\ge {n-2\choose k-2}+\gamma(\m F)$, we have $|\m F(\bar{1})|=\gamma(\m F)>0$ and
$|\m F'(1)|=|\m F(1)|\ge {n-2\choose k-2}+|\m F(\bar{1})|>{n-2\choose k-2}$. 
Given that $n>k$ and $\m F'(1)=\m L([2,n], |\m F(1)|, k-1)$, for any element $i\in [2,n]$, there exists $F\in \m F'$ such that $i\not\in F$. 
Since $\m F$ is non-trivial, $\m F'(\bar{1})=\m L([2,n], |\m F(\bar{1})|, k)\ne \emptyset$. 
Thus,  $\m F'$ is non-trivial. 
As ${n-1\choose k-1}\ge |\m F'(1)|=|\m F(1)|\ge {n-2\choose k-2}+|\m F(\bar{1})|={n-2\choose k-2}+|\m F'(\bar{1})|$, 
we deduce that $|\m F'(\bar{1})|\le {n-2\choose k-1}$. This implies that
for every set $F'\in \m F'(\bar{1})$,  we have $2\in F'$. Consequently, 
$\max_{i\in [2,n]}|\m F'(i)|\le |\m F'(2)|={n-2\choose k-2}+|\m F'(\bar{1})|\le |\m F'(1)|$. Thus, $\m F'_{\Delta}=\m F'(1)$.
\end{proof}

\begin{lemma}\label{1-21-1}
Let $n\ge2k\ge 6$, and let $\m A, \m B \subset {[n]\choose k}$ be non-trivial cross-intersecting families. If $|\m A|\ge h(n,k)$ and $\gamma(\m A)\le {n-3\choose k-2}$, then
$|\m A|+|\m B|\le 2h(n,k)$. 
\end{lemma}

\begin{proof}
We  may assume w.l.o.g. that $\m A_{\Delta}=\mathcal{A}(1)$. 
Using $|\m A|\ge h(n,k)$ together with \eqref{hnk>}, we have $|\m A|\ge {n-2\choose k-2}+2{n-3\choose k-2}$. Since $|\m A(\bar1)|=\gamma(\m A)\le {n-3\choose k-2}$,
\begin{equation}\label{a1>gamma}
    |\m A(1)|\ge{n-2\choose k-2}+{n-3\choose k-2} \ge{n-2\choose k-2}+\gamma(\m A).
\end{equation}

Since $\m A$ and $\m B $ are cross-intersecting, $\m A(1)$ and $\m B(\bar1)$ are cross-intersecting as well.
Note that $\m A(1)\subset {[2,n]\choose k-1}$ and $\m B(\bar1) \subset {[2,n]\choose k}$.
Since 
$|\m A(1)|\ge {n-2\choose k-2}+{n-3\choose k-2}$ and $n\ge 2k$, applying Proposition \ref{lowerupper}(ii) (with $n-1, k, k-1$ in the roles of $n, a,b$ and taking $i=1$),  
we have 
$$\gamma(\m B)\le |\m B(\bar1)|\le {n-3\choose k-2}.$$

Define $\m A_1, \m B_1$ as follows: $\mathcal{A}_1(1)=\mathcal{L}([2,n], |\mathcal{A}(1)|, k-1)$,  $\mathcal{A}_1(\bar{1})=\mathcal{L}([2,n], |\mathcal{A}(\bar{1})|, k)$, and similarly for $\m B_1$. 
By Theorem \ref{12-1-1}, $\mathcal{A}_1$ and $\mathcal{B}_1$ are cross-intersecting. 
By Proposition \ref{1-20-1} and (\ref{a1>gamma}), $\m A_1$ is non-trivial and ${\m A_1}_{\Delta}=\m A_1(1)$. 

Define $\m B_2$ as follows: $\m B_2(1)$ and $\m B_2(\bar1)$ are L-initial on $[2,n]$, 
$\m B_2(1)$ is maximal cross-intersecting with ${\m A_1}(\bar{1})$, and $\m B_2(\bar{1})=\m B_1(\bar{1})$. Then 
$\m A_1, \m B_2$ are cross-intersecting and $\m B_1\subset \m B_2$. 

\begin{claim}\label{b2nontrivial}
Family $\m B_2$ is non-trivial with ${\m B_2}_{\Delta}=\m B_2(1)$ and $\gamma(\m B_2)=|\m B_2(\bar1)|\le {n-3\choose k-2}$.
\end{claim}
\begin{proof}
Since $|\m A_1(\bar{1})|=|\m A(\bar{1})|\le {n-3\choose k-2}$, 
all sets in $\m A_1(\bar{1})$ contain both 2 and 3.
By the maximal choice of $\m B_2(1)$, $|\m B_2(1)|\ge {n-2\choose k-2}+{n-3\choose k-2}$.
As $|\m A_1(1)|=|\m A(1)|\ge {n-2\choose k-2}+{n-3\choose k-2}$, by the cross-intersecting property
of $\m A_1(1)$ and $\m B_2(\bar{1})$ and since $n\ge 2k$,
every set in $\m B_2(\bar{1})$ contains $\{2,3\}$, thus 
$\gamma(\m B_2)\le |\m B_2(\bar{1})|\le {n-3\choose k-2}$.
Then $|\m B_2(1)|\ge {n-2\choose k-2}+{n-3\choose k-2}$ gives
$|\m B_2(1)|\ge {n-2\choose k-2}+\gamma(\m B_2)$ and ${\m B_2}_{\Delta}=\m B_2(1)$. 
Since $\m B$ is non-trivial, $\m B(\bar 1)\ne \emptyset$, 
and hence $\m B_2(\bar 1)=\m B_1(\bar 1)\ne \emptyset$.
Thus, $\m B_2$ is non-trivial.
\end{proof}

Since $\m A_1(\bar1)$ and $\m B_2(\bar1)$ are L-initial on $[2,n]$, combining Claim \ref{b2nontrivial} and $|\m A_1(\bar 1)|=|\m A(\bar1)|\le {n-3\choose k-2}$, we conclude that every set in $\m A_1(\bar1)\cup \m B_2(\bar1)$
contains element 2. 
Let $\m A'=\{A\in \m A_1: 1\in A\}$ and $\m B'=\{B\in \m B_2: 1\in B\}$.
Then both $\m A' \cup (\m B_2(\bar 1))$ and $\m B' \cup (\m A_1(\bar 1))$ are $k$-uniform non-trivial intersecting families. 
By a result of Hilton and Milner (Theorem \ref{hm-nontrivial}), we have 
\begin{align*}
    &|\m A'|+ |\m B_2(\bar 1)|=|\m A' \cup (\m B_2(\bar 1))|\le h(n,k),\\
    &|\m B'|+ |\m A_1(\bar 1)|=|\m B' \cup (\m A_1(\bar 1))|\le h(n,k).
\end{align*}
Thus, $|\m A_1|+|\m B_2|=|\m A'|+ |\m B_2(\bar 1)|+|\m B'|+ |\m A_1(\bar 1)|\le 2h(n,k)$. 
Note that $|\m A|=|\m A_1|$ and $|\m B_2|\ge |\m B_1|=|\m B|$. 
We have $|\m A|+|\m B|\le 2h(n,k)$.
\end{proof}

\begin{proposition}\label{1-1-1}
Let $n\ge 2k\ge 6$, and let $\mathcal{A}, \mathcal{B} \subset {[n]\choose k}$ be non-trivial cross-intersecting families. 
Suppose that there exist distinct elements $i,j\in [n]$ such that $\mathcal{A}(\bar{i}\bar{j})=\mathcal{B}(\bar{i}\bar{j})=\emptyset$.
Then $|\m A|+|\m B|\le 2h(n,k)$.
\end{proposition}

\begin{proof}
Since $\mathcal{A}(\bar{i}\bar{j})=\mathcal{B}(\bar{i}\bar{j})=\emptyset$, every set in $\m A \cup \m B$ intersects $\{i,j\}$. 
Since $\m A$ and $\m B$ are non-trivial, $\mathcal{A}(i\bar{j}), \mathcal{A}(\bar{i}j), \mathcal{B}(i\bar{j}), \mathcal{B}(\bar{i}j)$ are non-empty families.
Moreover, $\mathcal{A}(i\bar{j}), \mathcal{B}(\bar{i}j)\subset {[n]\setminus \{i,j\}\choose k-1}$ are non-empty cross-intersecting. Since $n\ge 2k$, 
by Theorem \ref{12-3-2}, we get
\[
|\mathcal{A}(i\bar{j})|+|\mathcal{B}(\bar{i}j)|\le {n-2\choose k-1}-{n-1-k\choose k-1}+1.
\]
Similarly, 
\[
 |\mathcal{A}(\bar{i}j)|+|\mathcal{B}(i\bar{j})|\le {n-2\choose k-1}-{n-1-k\choose k-1}+1.
 \]
Combining these with $|\m A(\{i,j\})|,|\m B(\{i,j\})|\le{n-2\choose k-2}$, we obtain
\[
|\mathcal{A}|+|\mathcal{B}|\le 2{n-2\choose k-2}+2{n-2\choose k-1}-2{n-1-k\choose k-1}+2=2h(n,k).
\]
\end{proof}

\begin{proposition}\label{shifta}
Let $n\ge 2k\ge 6$, and let $\mathcal{A}, \mathcal{B} \subset {[n]\choose k}$ be non-trivial cross-intersecting families. 
If $S_{U,V}(\m A)$ is a star and $\gamma(\m A)>{n-3\choose k-2}$, then $|U|=|V|=1$.
\end{proposition}
\begin{proof}
Let $i$ be a center of $S_{U,V}(\mathcal{A})$. 
Then $S_{U,V}(\mathcal{A})(\bar{i})=\emptyset$.
Note that $|\mathcal{A}(\bar{i})|\ge \gamma(\mathcal{A})> {n-3\choose k-2}$. 
If $|U|=|V|\ge 2$, then since $i\in U$ (by Proposition \ref{1-21-2}), we move at most ${n-|U|-|V|\choose k-|U|}\le {n-4\choose k-2}$ sets to obtain $S_{U,V}(\mathcal{A})$ from $\m A(\bar i)$, which yields  
$|S_{U,V}(\mathcal{A})(\bar{i})|>{n-3\choose k-2}-{n-4\choose k-2}>0$, contradicting $S_{U,V}(\mathcal{A})(\bar{i})=\emptyset$.
\end{proof}

\begin{proposition}\label{12-31-1}
Let $n\ge 2k\ge 6$, and let $\mathcal{A}, \mathcal{B} \subset {[n]\choose k}$ be maximal non-trivial cross-intersecting families. 
Suppose further that $(\m A, \m B)$ is lex-minimal.
If $S_{U,V}^Q(\m A)$ is a star and $\gamma(\m A)>{n-3\choose k-2}$, 
then $|\m A|+|\m B|\le 2h(n,k)$.
\end{proposition}

\begin{proof}
Since $S_{U,V}^Q(\mathcal{A})$ is a star, by Proposition \ref{shifta}, $|U|=|V|=1$, and hence $S_{i,j}(\m A)$ is a star for some $1\le i<j \le n$. We may w.l.o.g. assume $\{i,j\}=\{1,2\}$.
Then, by Lemma \ref{12-22-1}, $\mathcal{A}(\bar{1}\bar{2})=\emptyset$ and $\mathcal{A}(1)\cap \mathcal{A}(2)=\emptyset$. Furthermore, since $\m A$ is non-trivial, $\mathcal{A}(1\bar{2})\ne \emptyset$ and $\mathcal{A}(2\bar{1})\ne \emptyset$.
By Proposition \ref{1-1-2}, all $k$-sets containing $\{1,2\}$ belong to $\mathcal{B}$, and both $\mathcal{B}(1\bar{2})\ne \emptyset$ and $\mathcal{B}(2\bar{1})\ne \emptyset$.

\begin{claim}\label{b12empty}
    $\mathcal{B}(\bar{1}\bar{2})=\emptyset$.
\end{claim}
\begin{proof}
Assume for contradiction that $\mathcal{B}(\bar{1}\bar{2})\ne \emptyset$. 
Since $n\ge 2k$ and $\m A$ and $\m B$ are cross-intersecting, $\{1,2\}$ is non-full in $\m A$, i.e., 
there exist $k$-sets $A\not\in \m A$ with $\{1,2\}\subset A \subset [n]$. 
Applying Lemma \ref{spl}(ii) with $r=1, R=\{1\}, i=2$, there exist a pair $(X, Y)$ and $S^{Q}_{X,Y}$-shift of $(\m A, \m B)$ such that $1\in X$ and $2\not\in Y$. 
We show that $S^{Q}_{X,Y}(\mathcal{A})$ and $S^{Q}_{X,Y}(\mathcal{B})$ are non-trivial, which contradicts Observation \ref{12-25-2}. 

Consider $S^{Q}_{X,Y}(\mathcal{A})$. 
Assume for contradiction that $S^{Q}_{X,Y}(\mathcal{A})$ is a star. 
By Proposition \ref{shifta}, $|X|=|Y|=1$, $X=\{1\}$, and $1$ is the unique center (by Proposition \ref{1-21-2}).  Since $2\not\in Y$, $S_{1,i}(\m A)$ is a star with center 1 and some $i\ge 3$.
Note that $\mathcal{A}(\bar1 \bar 2)=\emptyset$. 
We move at most ${n-3\choose k-2}$ sets to obtain $S_{1,i}(\m A)$ from $\mathcal{A}(2\bar{1})$. 
However, $|\mathcal{A}(2\bar{1})|\ge \gamma(\mathcal{A})> {n-3\choose k-2}$, this implies that $S_{1,i}(\mathcal{A})$ contains sets that do not contain 1, a contradiction. 
Thus, $S^{Q}_{X,Y}(\mathcal{A})$ is non-trivial. 

Consider $S^{Q}_{X,Y}(\m B)$. If $S^{Q}_{X,Y}(\m B)=\m B$, then we are done. Assume $S^{Q}_{X,Y}(\m B)\ne \m B$.
Since all $k$-sets containing $\{1,2\}$ are in $\mathcal{B}$ and $1\in X$, 
if $S^{Q}_{X,Y}(\mathcal{B})$ is a star, then its center must be $1$ or $2$. 
Note that $\m B(\bar12)\ne \emptyset$ and $\m B(1\bar 2)\ne \emptyset$.
Since $1\in X$ and $\m B(1\bar 2)\ne \emptyset$, $S^{Q}_{X,Y}(\mathcal{B})$ is not a star with center 2.
Note that $\{1,2\}$ is full in $\m B$ and $\m B(\bar12)\ne \emptyset$. 
Let $T\in \m B(\bar12)$, and $B=\{2\}\cup T$. Then $B\in \m B(\bar 1)$.
Since $1\in X$ and $2\not\in Y$, $\{1,2\}$ is full in $\m B$ implies that 
$S_{X,Y}(B)=B\in S^{Q}_{X,Y}(\mathcal{B})$ and $1\not\in B$. Thus
$S^{Q}_{X,Y}(\mathcal{B})$ is not a star with center 1. 
Thus $S^{Q}_{X,Y}(\mathcal{B})$ is non-trivial. 
This proves $\mathcal{B}(\bar{1}\bar{2})=\emptyset$.
\end{proof}

By Claim \ref{b12empty} and applying Proposition \ref{1-1-1} with $\{i,j\}=\{1,2\}$ yields $|\m A|+|\m B|\le 2h(n,k)$. 
\end{proof}

\subsection{Proof of Lemma \ref{mainlem1}}

\begin{proof}[Proof of Lemma \ref{mainlem1}]
We prove Lemma \ref{mainlem1} by arguing indirectly. Suppose that
Lemma \ref{mainlem1} is false. Among all counterexamples of
Lemma \ref{mainlem1}, choose $(\m A,\m B)$ maximizing
$|\m A|+|\m B|$.
Then 
$|\mathcal{A}|+|\mathcal{B}|\ge 2h(n,k).$
By symmetry, we may assume
$|\mathcal{A}|\ge h(n,k).$
By Lemma \ref{1-21-1}, we may assume
\begin{equation*}\label{12-19-2}
  \gamma(\mathcal{A})> {n-3\choose k-2}.
\end{equation*}
By Proposition \ref{12-31-1}, if there exists an $S_{U,V}$-shift of $(\m A, \m B)$ such that $S_{U,V}(\m A)$ is a star, then we are done. Thus, we may further assume that
\begin{equation}\label{notstar}
\text{For any $S_{U,V}$-shift of $(\m A, \m B)$, $S_{U,V}(\m A)$ is not a star.}    
\end{equation}
We may w.l.o.g. assume 
$$\m A_{\Delta}=\m A(1).$$

\begin{claim}\label{samemaxd}
    $\m B_{\Delta}=\m B(1).$
\end{claim}

\begin{proof}
Assume for contradiction that $\m B_{\Delta}=\m B(2)$, $|\m B(1)|<|\m B(2)|$ and $|\m A(1)|>|\m A(2)|$.
So $(\mathcal{A}, \mathcal{B})$ is not shifted; otherwise,  $1$ is the element of maximum degree in both families. Moreover, 
$(S_{1,2}(\mathcal{A}), S_{1,2}(\mathcal{B}))\precneqq(\mathcal{A},\mathcal{B})$. 
By Observation \ref{12-25-2} together with (\ref{notstar}), we conclude that $S_{1,2}(\mathcal{B})$ is a star. 
By Lemma \ref{12-22-1}, every set of $\mathcal{B}$ intersects  $\{1,2\}$, and hence $\m B(\bar 1\bar 2)=\emptyset$. 
Since $\mathcal{B}$ is non-trivial, $\mathcal{B}(1\bar{2})\ne \emptyset$ and $\mathcal{B}(2\bar{1})\ne \emptyset$.

Let $\m A_1, \m B_1$ be such that 
$\m A_1(\{1,2\})=\m L([3,n],|\m A(\{1,2\})|,k-2)$, 
$\m A_1(1\bar2)=\m L([3,n],|\m A(1\bar2)|,k-1)$,
$\m A_1(\bar12)=\m L([3,n],|\m A(\bar12)|,k-1)$ and
$\m A_1(\bar1\bar2)=\m L([3,n],|\m A(\bar1\bar2)|,k)$;
and similarly for $\m B_1$. 

By Theorem \ref{12-1-1}, $\m A_1, \m B_1$ are cross-intersecting. We show that both of them are non-trivial. 
Since $S_{1,2}(\m B)$ is a star, by Proposition \ref{1-1-2}, all $k$-sets containing $\{1,2\}$ belong to $\m A$, 
$\mathcal{A}(1\bar{2})\ne \emptyset$ and $\mathcal{A}(2\bar{1})\ne \emptyset$. 
Hence, all $k$-sets containing $\{1,2\}$ belong to $\m A_1$,  
$\mathcal{A}_1(1\bar{2})\ne \emptyset$ and $\mathcal{A}_1(2\bar{1})\ne \emptyset$. 
This implies that $\m A_1$ is not a star.
Suppose for contradiction that $\m B_1$ is a star with center $i$. 
Since $\mathcal{B}(1\bar{2})\ne \emptyset$ and $\mathcal{B}(\bar{1}2)\ne \emptyset$, 
$\mathcal{B}_1(1\bar{2})\ne \emptyset$ and $\mathcal{B}_1(\bar{1}2)\ne \emptyset$. Thus, $i\not\in \{1,2\}$. 
Thus $|\m B_1(2)|=|\m B_1(\{1,2,i\})|+|\m B_1(\{2,i\}\overline{\{1\}})|\le {n-3\choose k-3}+{n-3\choose k-2}={n-2\choose k-2}$. However, $|\m B_1(2)|=|\m B(2)|=\Delta(\m B)\ge {n-2\choose k-2}+1$, a contradiction. Thus, $\m A_1, \m B_1$ are non-trivial cross-intersecting families. 
Note that $|\m A_1||\m B_1|=|\m A||\m B|$. Then by the maximality of $|\m A||\m B|$, $\m A_1$ and $\m B_1$ are maximal non-trivial cross-intersecting families. 

Note that $|\m B(1)|<|\m B(2)|$ and $|\m A(1)|>|\m A(2)|$.
Thus, $|\m B_1(1)|<|\m B_1(2)|$ and $|\m A_1(1)|>|\m A_1(2)|$. 
Since $\m A_1(1\bar2), \m A_1(\bar12), \m B_1(1\bar2), \m B_1(\bar12)$ are L-initial on $[3,n]$, 
$\m A_1(\bar12) \subsetneqq \m A_1(1\bar2)$ and $\m B_1(1\bar2) \subsetneqq \m B_1(\bar12).$
Then there exists $B\in \m B_1(\bar12)$ such that $B\cup \{2\}\in \m B_1$ and $B\cup \{1\}\not\in \m B_1$. 
Since $\m A_1, \m B_1$ are maximal non-trivial cross-intersecting families, there exists $A\in \m A_1$ such that $A\cap (B\cup \{1\})=\emptyset$ and $A\cap (B\cup \{2\})\ne \emptyset$. This forces $1\not\in A$ and $2\in A$. 
Notice that $\m A_1(\bar12) \subsetneqq \m A_1(1\bar2)$, and hence $(A\setminus \{2\})\cup \{1\}\in \m A_1$.
This implies $((A\setminus \{2\})\cup \{1\})\cap(B\cup \{2\})=\emptyset$, a contradiction. 
\end{proof}

Let $\m A', \m B'$ be such that $\mathcal{A}'(1)=\mathcal{L}([2,n], |\mathcal{A}(1)|, k-1)$, $\mathcal{A}'(\bar{1})=\mathcal{L}([2,n], |\mathcal{A}(\bar{1})|, k)$, and similarly for $\mathcal{B}$.
From the definition, we conclude that
\begin{claim}\label{3-29-2}
$\m A', \m B'$ are non-trivial cross-intersecting families, and $\m A'(\bar 1 \bar 2)=\m B'(\bar 1 \bar 2)=\emptyset$.
\end{claim}

\begin{proof}
By Theorem \ref{12-1-1}, $\mathcal{A}'$ and $\mathcal{B}'$ are cross-intersecting. 
Note that $\m B_{\Delta}=\m B(1)$ (see Claim \ref{samemaxd}). Then
\begin{equation}\label{a'b'}
    |\mathcal{A}'(1)|=|\mathcal{A}(1)|=\Delta(\mathcal{A})\ge {n-2\choose k-2}+1 \,\,\text{and}\,\, |\mathcal{B}'(1)|=|\mathcal{B}(1)|=\Delta(\mathcal{B})\ge {n-2\choose k-2}+1.
\end{equation}
Since $n\ge 2k$, 
combining (\ref{a'b'}) with the cross-intersecting property of $\m A'$ and $\m B'$ yields 
$\m A'(\bar1 \bar 2)=\m B'(\bar1 \bar 2)=\emptyset$, as desired.
Since $\mathcal{A}'(1)$ and $\mathcal{B}'(1)$ are L-initial on $[2,n]$, it follows from (\ref{a'b'}) that 
\begin{itemize}
    \item[(i)] All $k$-sets containing $\{1,2\}$ belong to $\m A' \cap \m B'$;
    \item[(ii)] $\m A'(1\bar 2)\ne \emptyset$ and $\m B'(1\bar 2)\ne \emptyset$. 
\end{itemize}
As $\m A, \m B$ are non-trivial, $|\m A'(\bar 1)|=|\m A(\bar1)|\ge 1$ and $|\m B'(\bar 1)|=|\m B(\bar1)|\ge 1$.
By (i), if $\m A'$ or $\m B'$ were a star, then its center must belong to $\{1,2\}$.
Combining this with (i) and (ii), we conclude that $\m A'$ and $\m B'$ are non-trivial. 
\end{proof}

Combining Claim \ref{3-29-2} and Proposition \ref{1-1-1}, we get $|\mathcal{A}|+|\mathcal{B}|=|\mathcal{A}'|+|\mathcal{B}'|\le 2h(n,k)$.
We complete the proof of Lemma \ref{mainlem1}.
\end{proof}

\section{Degree-asymmetric case: Proof of Lemma \ref{mainlem2}}\label{sec:asymmetric}
We now turn to Lemma~\ref{mainlem2}, the asymmetric maximum-degree case. The
proof requires a finer analysis of L-initial families. We begin
by recalling the partner construction and a result of Huang--Peng, which gives a precise way to pair two
L-initial families so that they are cross-intersecting, and it
also gives maximality on one side \cite{k_1+k_3}. 

Let $F, H \subset [n]$ satisfy that $\max F=\max H=q$, $|F|=f$, $|H|=h$, $F\cap H=\{q\}$ and $F\cup H=[q]$. Let $k$ be such that $k\leq n-f$. 
We define the {\it $k$-partner} $K$ of $F$ such that $|K|=k$, $K\prec H$, and there is no other $k$-set $K'$ satisfying $K\precneqq K'\prec H$. By the definition, we see that $K=H$ if $k=h$; $K= H\cup [n-k+h+1, n]$ if $k>h$. 
Let $F\in {[n]\choose k}$. We write $\m L([n], F, k):=\{A\in {[n]\choose k}: A\prec F\}$.

\begin{lemma}[Huang--Peng \cite{k_1+k_3}]\label{3-12-1}
Let $n\ge a+b$, $A\subset [n]$, $|A|=a$, and let $B$ be the $b$-partner of $A$. Then
$\mathcal{L}([n], B, b)$ and $\mathcal{L}([n], A, a)$ are cross-intersecting. Moreover, $\mathcal{L}([n], B, b)$ is maximal cross-intersecting with $\mathcal{L}([n], A, a)$.
\end{lemma}

We shall use the following result of Frankl and Wang, which gives a lower bound on the size of each family in an extremal product pair.

\begin{proposition}[Frankl--Wang \cite{11-6-1}]\label{1-29-1}
Let $n\ge 2k\ge 6$, and let $\mathcal{A}, \mathcal{B} \subset {[n]\choose k}$ be non-trivial cross-intersecting families such that $|\m A||\m B|$ is maximum among all pairs of non-trivial  cross-intersecting families.  Then 
\begin{equation}\label{minAB}
\min \{|\m A|, |\m B|\}>{n-3\choose k-3}+{n-4\choose k-3}.    
\end{equation}
\end{proposition}

We then present several auxiliary constructions and compute the products of their sizes, which will be used to exclude near-extremal configurations.

\begin{construction}\label{2-27-1}
 Let $k\ge 3$ and $n\ge 2k$ be positive integers. We define 
    \begin{align*}
        &\m C_1=\left\{C\in {[n]\choose k}: \{1,2\}\subset C\right\} \cup \left\{[3,k+2]\right\},\\
        &\m C_2=\left\{C\in {[n]\choose k}: \{i,j\}\subset C, i\in [2],j\in [3,k+2]\right\}.
    \end{align*}
\end{construction}

\begin{proposition}\label{3-11-2}
Let $n\ge 2k$ and $k\ge 3$. Then $ |\m C_1||\m C_2|<h(n,k)^2$.
\end{proposition}

\begin{proof}
A direct calculation gives
\begin{equation*}
|\m C_1|={n-2\choose k-2}+1,\,\,|\m C_2|={n-2\choose k-2}+2{n-3\choose k-2}+\dots +2{n-k-1\choose k-2}+{n-k-2\choose k-2}. 
\end{equation*}
Let 
\begin{align*}
    &x:={n-3\choose k-2}+{n-4\choose k-2}+\dots +{n-k-1\choose k-2},\\
    &y:={n-3\choose k-2}+{n-4\choose k-2}+\dots +{n-k-1\choose k-2}+{n-k-2\choose k-2}-1.
\end{align*}
Then $|\m C_1|+x>y$. 
Since $n\ge 2k$ and $k\ge 3$, 
$
|\m C_1|={n-2\choose k-2}+1=\sum_{i=1}^{k-1}{n-2-i\choose k-1-i}+1\le \sum_{i=1}^{k-1}{n-2-i\choose k-2}=x.
$
Therefore, $|\m C_1|/x\le 1<(|\m C_1|+x)/y$. 
Note that $h(n,k)=|\m C_1|+x=|\m C_2|-y$.
Hence
$
\frac{|\m C_2|}{y}
=
\frac{|\m C_1|+x+y}{y}
=
\frac{|\m C_1|+x}{y}+1
=
\frac{|\m C_1|+x}{y}+\frac{x}{x}
>
\frac{|\m C_1|+x}{x}.
$
So $|\m C_2|x>(|\m C_1|+x)y$. This implies $h(n,k)^2=(|\m C_1|+x)(|\m C_2|-y)>|\m C_1||\m C_2|$. 
\end{proof}

\begin{construction}\label{H_AH_B}
Let $B\in {[n]\choose k}$ and $\{1,2,4,n-k+4, \dots, n\}\prec B \prec \{1,2, n-k+3, \dots,n\}$, and let $A$ be the $k$-partner of $B$. Define 
\begin{align*}
    &\m H_A=\{H\in \tbinom{[n]}{k}: H\cap \{1,2\}\ne \emptyset, H\cap [3,k+2]\ne \emptyset \}\cup \m L([3,n], A, k),\\
    &\m H_B=\m L([n], B, k) \cup \{[3,k+2]\},
\end{align*}
where if $B=\{1,2, n-k+3,\dots, n\}$, then $A=\{2,n-k+2,\dots,n\}$, and in this case we define
$\m L([3,n],\{2,n-k+2,\dots,n\},k)$ to be the empty family.
\end{construction}

Note that for $B=\{1,2, n-k+3,\dots, n\}$, we have $\m H_A=\m C_1$ and $\m H_B=\m C_2$. Proposition \ref{3-11-2} states that $|\m C_1||\m C_2|<h(n,k)^2$. Furthermore, we have the following proposition. 
\begin{proposition}\label{1-22-5}
Let $n\ge 2k\ge 6$, $B\in {[n]\choose k}$ and $\{1,2,4,n-k+4, \dots, n\}\prec B \prec \{1,2, n-k+3, \dots,n\}$, and let $A$ be the $k$-partner of $B$. Then 
$|\m H_A||\m H_B|<h(n,k)^2.$
\end{proposition}

\begin{proof}
Since $\{1,2,4,n-k+4, \dots, n\}\prec B \prec \{1,2, n-k+3, \dots,n\}$, $3\in A$. 
By Lemma \ref{3-12-1}, $\m L([n], B, k)$ and $\m L([3,n], A, k)$ are cross-intersecting. Clearly, the other part $\m H_A\setminus \m L([3,n], A, k)$ is cross-intersecting with $\m H_B$. 
Thus, $\m H_A$ and $\m H_B$ are non-trivial cross-intersecting families in the above constructions.

 We define
\begin{align*}
&\m H'_B:=\{H: |H|=k, \{1,2\}\subset H\subset [n], \{3,4\}\cap H\ne \emptyset\},\\
&\m H''_B:=\m H_B\setminus (\m H'_B \cup \{[3,k+2]\}).
\end{align*}
Since  $\{1,2,4,n-k+4, \dots, n\}\prec B$ and $\m L([n], B, k)\subset \m H_B$,
$\m H'_B\subset \m H_B$.
Note that $\{1,2\}\subset B$, $\{3,4\}\subset A$ and for any $B'\in \m H''_B$, we have  $B'\cap \{3,4\}=\emptyset$. 
Thus, 
\begin{align*}
&|\m L([3,n], A, k)|=|\m L([5,n], |\m L([3,n], A, k)|, k-2)|=:\alpha_1,\\
&|\m H''_B|=|\m H''_B([2])|=|\m L([5,n], |\m H''_B|, k-2)|=:\alpha_2.
\end{align*}

We claim 
\begin{equation}\label{3-12-3}
\alpha_1 \le {n-4\choose k-2}-\alpha_2.
\end{equation}
Indeed, by the cross-intersecting property of $\m L([n], B, k)$ and $\m L([3,n], A, k)$,  $\m L([5,n], |\m H''_B|, k-2)$ and 
$\m L([5,n], |\m L([3,n], A, k)|, k-2)$ are also cross-intersecting. Then by Hilton's result Theorem \ref{3-12-2}, we have $|\m L([5,n], |\m L([3,n], A, k)|, k-2)|+|\m L([5,n], |\m H''_B|, k-2)|=\alpha_1+\alpha_2 \le {n-4\choose k-2}$. So $\alpha_1 \le {n-4\choose k-2}-\alpha_2$, as claimed. 

We write 
\begin{align*}
    &\beta_1={n-1\choose k-1}+{n-2\choose k-1}-{n-1-k\choose k-1}-{n-2-k\choose k-1},\\
    &\beta_2=1+{n-3\choose k-3}+{n-4\choose k-3}.
\end{align*}
So $|\m H_A|=\beta_1+\alpha_1$ and $|\m H_B|=1+|\m H'_B|+|\m L([5,n], |\m H''_B|, k-2)|=\beta_2+\alpha_2$. 
Recall that $|\m C_1||\m C_2|<h(n,k)^2$. To prove Proposition \ref{1-22-5}, it is sufficient to show $|\m H_A||\m H_B|\le |\m C_1||\m C_2|$. 
Clearly, $|\m C_1|=\beta_2+{n-4\choose k-2}$, $|\m C_2|=\beta_1$ and 
\begin{equation}\label{alpha2}
  |\m C_1|-(\tbinom{n-4}{k-2}-\alpha_2)<|\m C_2|.
\end{equation}
 Then
\begin{align*}
|\m C_1||\m C_2|-|\m H_A||\m H_B|
&=|\m C_1||\m C_2|-(\alpha_1+\beta_1)(\alpha_2+\beta_2)\\
&=|\m C_1||\m C_2|-(\alpha_1+|\m C_2|)(\alpha_2+|\m C_1|-\tbinom{n-4}{k-2})\\
&=|\m C_2|(\tbinom{n-4}{k-2}-\alpha_2)-\alpha_1(\alpha_2+|\m C_1|-\tbinom{n-4}{k-2})\\
&\overset{(\ref{3-12-3}),(\ref{alpha2})}{>}0.
\end{align*}
\end{proof}

\begin{construction}\label{2-27-2}
Let $n\ge 8$. We define
    \begin{align*}
    &\m D_1=\left\{D\in {[n]\choose 4}: \{i,j\}\subset D, i\in \{2,3,4\},j\in \{1,5,6\}\right\},\\
    &\m D_2=\left\{D\in {[n]\choose 4}: \{1,5,6\}\subset D\right\} \cup \left\{D\in {[n]\choose 4}: \{2,3,4\}\subset D\right\}.
    \end{align*}
\end{construction}
By a direct computation, we have
$|\m D_2|=2(n-3)$ and 
$|\m D_1|={n-2\choose 2}+2{n-3\choose 2}+3{n-4\choose 2}+2{n-5\choose 2}+{n-6\choose 2}$.
Then
$|\m D_1||\m D_2|=(2n-6)\frac{9n^2-81n+192}{2}$. 
Note that $h(n,4)=2n^2-16n+35$.
Using $n\ge8$, we obtain
\begin{equation}\label{2-27-4}
    |\m D_1||\m D_2|<h(n,4)^2.
\end{equation}

\subsection{Proof of Lemma \ref{mainlem2}}

We begin by introducing the following assumption, which somewhat simplifies the proof of Lemma~\ref{mainlem2}.
\begin{assumption}\label{asm}
Let $n\ge 2k\ge 6$, and let $\mathcal{A}, \mathcal{B} \subset {[n]\choose k}$ be non-trivial cross-intersecting families such that $|\m A||\m B|$ is maximum among all pairs of non-trivial cross-intersecting families with  
$\Delta(\m A)\ge {n-2\choose k-2}+1$ and 
$\Delta(\m B)\le {n-2\choose k-2}$. Moreover, 
\begin{itemize}
\item $|\m A|\ge h(n,k)$, $\gamma(\m A)>{n-3\choose k-2}$;
\item Any $S^{Q}_{U,V}$-shift of $(\m A, \m B)$ makes $S_{U,V}^Q(\m B)$ a star, and keeps $S_{U,V}^Q(\m A)$ non-trivial.
\end{itemize}
\end{assumption}

\begin{proposition}\label{covernumber1}
Under Assumption \ref{asm}, if, in addition, $\tau(\m B_{\Delta})=1$, then $\gamma(\m B)=1$.
\end{proposition}

\begin{proof}
We may w.l.o.g. assume $\mathcal B_\Delta=\mathcal B(1).$
Since $\tau(\mathcal B_\Delta)=1$, we may, furthermore, assume that $2$ is a common
element of all members of $\mathcal B(1)$. Thus
$\mathcal B(1\bar2)=\emptyset$ and
$|\mathcal B(1,2)|=|\mathcal B(1)|=\Delta(\mathcal B)$.
Since $|\mathcal B(2)|\le \Delta(\mathcal B)$, it follows that
$\mathcal B(\bar1 2)=\emptyset$. Hence
$\mathcal B_\gamma=\mathcal B(\bar1)=\mathcal B(\bar1\bar2)
\subset {[3,n]\choose k}.$

Note that any $S^{Q}_{U,V}$-shift of $(\m A, \m B)$ makes $S_{U,V}^Q(\m B)$ a star. 
Thus, for any $i$, if $i\in B$ for some $B\in \m B(\bar1\bar2)$, then $S_{1,i}(\m B)$ is a star, and $1$ is a center of  $S_{1,i}(\m B)$ (by  Proposition \ref{1-21-2}). 
Suppose $\gamma(\m B)\ge 2$. Then there are two distinct sets $B_1$ and $B_2$ in $\m B(\bar 1 \bar 2)$.
Choose $j\in B_1\setminus B_2$. 
Clearly, $B_2\in S_{1,j}(\m B)$ and $1\not\in B_2$,  
contradicting the fact that $S_{1,j}(\mathcal B)$ is a star with center $1$. Thus, $\gamma(\m B)=1$. 
\end{proof}

\begin{proposition}\label{covernumber2-shifted}
Under Assumption \ref{asm}, if, in addition,  $(\mathcal A,\mathcal B)$ is shifted, then $\gamma(\m B)=1$.
\end{proposition}
\begin{proof}
Suppose, for contradiction, that $\gamma(\m B)\ge 2$. 
By Proposition \ref{covernumber1}, we have $\tau(\m B_{\Delta})\ge 2$. 
Write $M:=\cap_{B\in \m B(\bar{1})}B.$ 
Since $\m A, \m B$ are shifted, $\m A_{\Delta}=\m A(1)$, $\m B_{\Delta}=\m B(1)$ and $M=[2,m]$ for some $m$, allowing $m=1$ when $M=\emptyset$. Since $|\m B(\bar1)|=\gamma(\m B)\ge 2$, $m\le k$.
By the maximality of $|\m A||\m B|$, $\m A$ and $\m B$ are maximal cross-intersecting families.
Then for each $i\in [2,m]$, $\{1,i\}$ is full in $\m A$. 

We show that $\{1,m+1\}$ is non-full in $\m A$. 
Since $m+1\not\in M$, there exists $B\in \m B(\bar1)$ such that $B\cap \{1,m+1\}=\emptyset$. 
Then $B\setminus M \subset [m+2,n]$ and $|B\setminus M|=k-m+1$.
As $n\ge 2k$, $|[m+2,n]\setminus (B\setminus M) |=n-k-2\ge k-2$. 
Then there exists a $(k-2)$-set $T \subset [m+2,n]\setminus (B\setminus M)$. 
Then $\{1,m+1\}\cup T \in \mathcal E^k_{\{1,m+1\}}$. 
By the cross-intersecting property of $\m A$ and $\m B$, $\{1,m+1\}\cup T \not\in \m A$. 
Thus, $\{1,m+1\}$ is non-full in $\m A$.

As $\m A$ and $\m B$ are maximal cross-intersecting, if every set in $\m A(\bar 1)$ contains $m+1$, then all $k$-sets containing $\{1,m+1\}$ belong to $\m B$, i.e., $|\m B(\{1,m+1\})|={n-2\choose k-2}$, together with $|\m B(1)|=\Delta(\m B)\le {n-2\choose k-2}$, we see that every set in $\m B(1)$ contains $m+1$, contradicting $\tau(\m B_{\Delta})\ge 2$. 
Thus, there exists $A'\in \m A$ such that $A'\cap \{1,m+1\}=\emptyset$.

Apply Lemma \ref{spl}(i) to $\m A$ with $X=[n]$, $r=1$ and $R=\{1, m+1\}$.
The required conditions are verified as follows: (P1) holds since $|R|=2$ and $\m A$ is shifted; (P2) holds by the existence of $A'$; (P3) holds since  $\{1,m+1\}$ is non-full in $\m A$. 
Thus, by Lemma \ref{spl}(i),
there exists an $S_{U,V}^Q$-shift of $\m A$ (and hence of $(\m A, \m B)$) such that $\{1,m+1\}\subset U$. 
By Assumption \ref{asm},  $S_{U,V}^Q(\m B)$ is a star. Moreover, by Proposition \ref{1-21-2}, $S_{U,V}^Q(\m B)$ is a star with center $c\in U$. Since $1\in U$, $\m B[1]\subset S_{U,V}^Q(\m B)$.
We claim $c=1$, since if not, then every set in $\m B[1]$ contains $c$, contradicting $\tau(\m B(1))=\tau(\m B_{\Delta})\ge 2$. 
However, since $\m B$ is shifted, $[2,k+1]\in \m B$. 
As $m\le k$, $m+1\in [2,k+1]$. Then $m+1\in U$ implies that $[2,k+1]$ is stable and $[2,k+1]\in S_{U,V}^Q(\m B)$, which contradicts the fact that $S_{U,V}^Q(\m B)$ is a star with center $1$.
\end{proof}

The following two propositions constitute the key ingredients
in the proof of Lemma~\ref{mainlem2}. We defer their proofs
to the next two subsections.

\begin{proposition}\label{gamma1}
Let $n\ge 2k\ge6$, $\m A, \m B\subset {[n]\choose k}$ be non-trivial cross-intersecting with $\Delta(\m B)\le {n-2\choose k-2}$, $\gamma(\m B)=1$ and $|\m B|>{n-3\choose k-3}+{n-4\choose k-3}$. Then the following hold.
\begin{itemize}
    \item[(i)] If $\tau(\m B_{\Delta})=1$, then $|\m A||\m B|<h(n,k)^2$.  
    \item[(ii)] If $(\mathcal A,\mathcal B)$ is shifted, then $|\m A||\m B|<h(n,k)^2$. 
\end{itemize} 
\end{proposition}

The following proposition deals with the case in which $\tau(\mathcal B_\Delta)\ge2$, and $(\m A, \m B)$ is not shifted. In this case we
shall use one additional normalization: the elements $2,\ldots,n$ are
ordered by their degrees in $\mathcal B(\bar1)$. This loses no generality
because in the proof of Lemma \ref{mainlem2} we may relabel these elements
before applying the proposition.

\begin{proposition}\label{cover2-nonshifted}
Under Assumption \ref{asm}, suppose, in addition, that
$\tau(\m B_{\Delta})\ge 2$, and $(\mathcal A,\mathcal B)$ is not shifted.
Suppose further that the elements $2,3,\ldots,n$ are labeled so that
\begin{equation}\label{degree-normalization}
d_{\mathcal B(\bar1)}(2)\ge d_{\mathcal B(\bar1)}(3)
\ge \cdots \ge d_{\mathcal B(\bar1)}(n).
\end{equation}
Then
$|\m A||\m B|< h(n,k)^2$.
\end{proposition}

The proof of Proposition \ref{cover2-nonshifted} is the most technical part
of the argument. We postpone it to the next subsection. Assuming
Proposition \ref{cover2-nonshifted}, we now complete the proof of
Lemma \ref{mainlem2}.

\begin{proof}[Proof of Lemma \ref{mainlem2}]
Suppose, for contradiction, that $|\mathcal A||\mathcal B|\ge h(n,k)^2.$ 
Moreover, we may take $(\m A, \m B)$ to be the lex-minimal pair, and w.l.o.g. assume $\mathcal B_\Delta=\mathcal B(1).$
If $\gamma(\m A)\le {n-3\choose k-2}$, then since $|\m A|\ge h(n,k)$, by Lemma \ref{1-21-1}, $|\m A|+|\m B|\le 2h(n,k)$, and hence $|\mathcal A||\mathcal B|\le h(n,k)^2$, as required. 
Next, we may assume $\gamma(\m A)> {n-3\choose k-2}$. 
Since $(\m A, \m B)$ is lex-minimal, by Observation \ref{12-25-2}, for any $S_{U,V}^Q$-shift of $(\m A, \m B)$, at least one of $S_{U,V}^Q(\m A)$ and $S_{U,V}^Q(\m B)$ is a star. If $S_{U,V}^Q(\m A)$ is a star for some $S_{U,V}^Q$-shift, then by Proposition \ref{12-31-1}, we have $|\m A|+|\m B|\le 2h(n,k)$, and hence $|\mathcal A||\mathcal B|\le h(n,k)^2$, as required. Thus, we may assume that for any $S_{U,V}^Q$-shift of $(\m A, \m B)$, $S_{U,V}^Q(\m B)$ is a star, and $S_{U,V}^Q(\m A)$ keeps non-trivial. 
Consequently, we may assume that
$(\mathcal A,\mathcal B)$ satisfies Assumption \ref{asm}. 
If $\tau(\m B_{\Delta})=1$ or $(\m A, \m B)$ is shifted, then combining Propositions (\ref{covernumber1}) and (\ref{covernumber2-shifted}), yields $\gamma(\m B)=1$. 
By (\ref{minAB}), we have $|\m B|>{n-3\choose k-3}+{n-4\choose k-3}$. Together with Proposition \ref{gamma1}, this yields $|\m A||\m B|<h(n,k)^2$, a contradiction. Thus $\tau(\m B_{\Delta})\ge 2$, and $(\mathcal A,\mathcal B)$ is not shifted. 
In what follows, we reorder the elements of $[2,n]$ in decreasing order of degrees in $\m B(\bar 1)$, and denote by $(\m A', \m B')$ the resulting pair of families. So 
\begin{equation}\label{6-9-1}
d_{\mathcal B'(\bar1)}(2)\ge d_{\mathcal B'(\bar1)}(3)
\ge \cdots \ge d_{\mathcal B'(\bar1)}(n).
\end{equation}

We show that we can reduce $(\mathcal A',\mathcal B')$ to satisfy the condition of Proposition \ref{cover2-nonshifted}. 
Clearly, $|\mathcal A'|=|\mathcal A|$, $|\mathcal B'|=|\mathcal B|$, $|\mathcal B'(1)|=|\mathcal B(1)|$
and $(\mathcal A',\mathcal B')$ is still a pair of non-trivial
cross-intersecting $k$-uniform families. Moreover, $\mathcal B'_\Delta=\mathcal B'(1)$, $\Delta(\mathcal B')=|\mathcal B'(1)|
=\Delta(\mathcal B)\le {n-2\choose k-2}$, and $\gamma(\mathcal B')=\gamma(\mathcal B)$, $\tau((\mathcal B')_\Delta)=\tau(\mathcal B_\Delta)\ge2$.
Assume on the contrary that $(\mathcal A',\mathcal B')$ does not satisfy the condition of Proposition \ref{cover2-nonshifted}. If $(\mathcal A',\mathcal B')$ satisfies Assumption \ref{asm}, then $(\mathcal A',\mathcal B')$ is shifted, and hence $|\m A||\m B|=|\m A'||\m B'|<h(n,k)^2$ by Proposition \ref{gamma1}, a contradiction. Thus $(\mathcal A',\mathcal B')$ does not satisfy Assumption \ref{asm}. Using the same reduction used to obtain Assumption \ref{asm} applies to  $(\mathcal A',\mathcal B')$. Since
$|\mathcal A'||\mathcal B'|=|\mathcal A||\mathcal B|,$
this would already give $|\mathcal A|+|\mathcal B|\le 2h(n,k),$ and hence $|\mathcal A||\mathcal B|\le h(n,k)^2,$ as required. Hence we may assume that
$(\mathcal A',\mathcal B')$ satisfies Assumption \ref{asm}. Thus, we may assume that $(\mathcal A',\mathcal B')$ satisfy the condition of Proposition \ref{cover2-nonshifted}. 

Applying Proposition \ref{cover2-nonshifted} to $(\m A', \m B')$, 
we have $|\m A||\m B|=|\m A'||\m B'|< h(n,k)^2$, contradicting our assumption. This proves Lemma \ref{mainlem2}.
\end{proof}

From the proof of Lemma \ref{mainlem2}, the following corollary holds immediately.
\begin{corollary}\label{mainlem2-coro}
Let $n\ge 2k\ge 6$, and let $\mathcal{A}, \mathcal{B} \subset {[n]\choose k}$ be non-trivial cross-intersecting families such that $|\m A||\m B|$ is maximum among all pairs of non-trivial  cross-intersecting families.  
Suppose further that $(\m A, \m B)$ is lex-minimal, $|\m A|\ge h(n,k)$, $\Delta(\m A)\ge {n-2\choose k-2}+1$ and $\Delta(\m B)\le {n-2\choose k-2}$. Then either $|\m A||\m B|< h(n,k)^2$ or $|\m A|+|\m B|\le 2h(n,k)$. 
\end{corollary}

\subsection{Proof of Proposition \ref{gamma1} }
\begin{proof}[Proof of Proposition \ref{gamma1}]
Among all the pairs of families that satisfy the condition of Proposition \ref{gamma1}, we may choose $\m A, \m B$ with maximum $|\m A||\m B|$. 
When we prove (i), we may w.l.o.g. assume $\m B_{\Delta}=\m B(1)$; when we prove (ii), since $\m B$ is shifted,  $\m B_{\Delta}=\m B(1)$. Thus, in the rest of the proof, we may always set $\m B_{\Delta}=\m B(1)$. 

Note that ${n-3\choose k-3}+{n-4\choose k-3}\le |\mathcal B(1)|\le {n-2\choose k-2}.$
Choose the $k$-set $B^\ast$ such that $|\mathcal L([n],B^\ast,k)|=|\mathcal B(1)|.$
Then $\{1,2,4,n-k+4,\ldots,n\}\prec B^\ast
\prec \{1,2,n-k+3,\ldots,n\}.$
Let $A^\ast$ be the $k$-partner of $B^\ast$, and let
$(\mathcal H_{A^\ast},\mathcal H_{B^\ast})$ be defined in  Construction \ref{H_AH_B}. 
By construction, $|\mathcal H_{B^\ast}|=|\mathcal B|.$
By Proposition \ref{1-22-5}, we have
$|\m H_{A^\ast}||\m H_{B^\ast}|<h(n,k)^2$. 
Since $\m B_{\Delta}=\m B(1)$ and
$\gamma(\mathcal B)=1$, there is a unique member of $\mathcal B$ not
containing $1$. Denote this unique set by $B_0$; thus $\m B(\bar1)=\{B_0\}$.  
By the maximality of $|\m A||\m B|$,
\begin{equation}\label{6-17-3}
\text{all $k$-sets $A$ with $1\in A$ and $A\cap B_0\ne \emptyset$ belong to $\m A$.}
\end{equation}
Thus
$|\m A(1)|=|\m H_{A^\ast}(1)|$. 
Thus to prove Proposition \ref{gamma1}, it is enough to prove
\begin{equation}\label{gamma1-A0}
|\mathcal A(\bar1)|\le |\m H_{A^\ast}(\bar1)|.
\end{equation}

We first prove (i). 
We may relabel $[2,n]$ such that $B_0=[3,k+2].$
Thus $\mathcal B(\bar1\bar2)=\{[3,k+2]\}.$ 
In this case we may w.l.o.g. assume that every set in $\m B(1)$ contains $2$, i.e., $\m B[1]=\m B[\{1,2\}]$. 
By the maximality of $|\m A||\m B|$, we see that
all $k$-sets $A$ with $2\in A$ and $A\cap B_0\ne \emptyset$ belong to $\m A$. Thus
$\m A(\bar1 2)=\m H_{A^\ast}(\bar12)$. 
Since $\m A(\bar1\bar2)$ and $\m B(12)$ are cross-intersecting, 
$\m L([3,n],|\m A(\bar1\bar2)|,k)$ and $\mathcal L([n],B^\ast,k)$
are also cross-intersecting.
By the definition of the $k$-partner and by Lemma \ref{3-12-1},
$\mathcal L([3,n],A^\ast,k)$ is maximal cross-intersecting with
$\mathcal L([3,n],|\m B(1)|,k-2)$ (note that $|\m B(1)|\le {n-2\choose k-2}$). Thus $|\m A(\bar1\bar2)|=|\m L([3,n],|\m A(\bar1\bar2)|,k)|\le |\m H_{A^\ast}(\bar1\bar2)|$. 
Together with $\m A(\bar1 2)=\m H_{A^\ast}(\bar12)$, this yields $|\m A(\bar1)|\le |\m H_{A^\ast}(\bar1)|$. This proves (i).

We now prove (ii). In this case $(\mathcal A,\mathcal B)$ is shifted, and $\m B_{\Delta}=\m B(1)$. Since $\gamma(\m B)=1$, $\m B(\bar{1})=\{[2,k+1]\}$. 
We interchange the labels $2$ and $k+2$, so that $[2,k+1]$ becomes $[3,k+2]$;
that is, we relabel $2$ as $k+2$ and $k+2$ as $2$, leaving all other labels unchanged; we denote the resulting pair of families by $\m A'$ and $\m B'$. 
Then $\m B'(\bar1)=\{B_0\}$, and $\m A', \m B'$ are also non-trivial cross-intersecting families with $|\m A'||\m B'|=|\m A||\m B|$, $|\m B'(1)|=|\m B(1)|\le{n-2\choose k-2}$, $\gamma(\m B')=1$, and $|\m B'|=|\m B|>{n-3\choose k-3}+{n-4\choose k-3}$.
If $\tau(\m B'_{\Delta})=1$, then by (i), we are done. We may next assume $\tau(\m B'_{\Delta})\ge 2$.

\begin{claim}\label{FF'}
Let $\m F\in \{\m A, \m B\}$, and let $\m F'=\m A'$ if $\m F=\m A$, or $\m F'=\m B'$ if $\m F=\m B$. For any $F\in \m F$, if $F\notin \m F'$, then $2\in F$, $k+2\notin F$, and $(F\setminus \{2\})\cup \{k+2\}\notin \m F$.
\end{claim}

\begin{proof}
Let $F\in \m F$. We define $F'$ as follows: if $F\cap \{2,k+2\}=\{2\}$, then $F'=(F\setminus \{2\})\cup \{k+2\}$; if $F\cap \{2,k+2\}=\{k+2\}$, then $F'=(F\setminus \{k+2\})\cup \{2\}$; otherwise $F'=F$.
From the definition of $\m F'$ and $F'$, $F\notin \m F'$ implies that $|F\cap \{2,k+2\}|=1$. 
Since $F\notin \m F'$, $(F\setminus \{2,k+2\})\cup \{2\}$ and $(F\setminus \{2,k+2\})\cup \{k+2\}$ cannot both belong to $\m F$. 
Since $\m A, \m B$ are shifted and $\m F\in \{\m A, \m B\}$, $\m F$ is shifted.
Thus, if $F\cap \{2,k+2\}=\{k+2\}$, then $F'\in \m F$, a contradiction. Thus $F\cap \{2,k+2\}=\{2\}$, and $(F\setminus \{2\})\cup \{k+2\}\notin \m F$.
\end{proof}

\begin{claim}\label{ijstable}
If $i<j$ and either $i\ne2$ and $j\notin[3,k+2],$ or $i=2$ and $j>k+2,$ 
then $S_{i,j}$ acts trivially on
$(\mathcal A',\mathcal B')$.
\end{claim}

\begin{proof}
Let $\sigma$ be the transposition of $2$ and $k+2$. Thus
$\m A'=\sigma(\m A)$ and $\m B'=\sigma(\m B)$.
Take an arbitrary $\m F\in\{\m A,\m B\}$, and write
$\m F':=\sigma(\m F)$. Then $\m F$ is
shifted since $(\m A,\m B)$ is shifted.
Let $i<j$ with $i\ne 2$ and $j\notin [3,k+2]$,  or $i=2$ and $j>k+2$. 

We show that
$\sigma(i)<\sigma(j).$
If $i\in \{2,k+2\}$, then $j>k+2,$ and hence $\sigma(i)=k+2<j=\sigma(j)$. 
If $j=2$, then $i=1$, and hence
$\sigma(i)=1<k+2=\sigma(j).$ Assume $j\ne 2$.
Since $j\notin[3,k+2]$ and $i<j$, we have $j>k+2$. If
$i=k+2$, then $\sigma(i)=2<j=\sigma(j)$; 
if $i\ne k+2$, then since $i\ne 2$, we have $\sigma(i)\le i<j=\sigma(j)$. Thus in all cases $\sigma(i)<\sigma(j)$.

Now take an arbitrary $S\in\m F'$ with $j\in S$ and $i\notin S$. Put
$T:=\sigma(S).$
Then $T\in\m F$, $\sigma(j)\in T$, and $\sigma(i)\notin T$. Since
$\m F$ is shifted and $\sigma(i)<\sigma(j)$, we have
$T':=(T\setminus\{\sigma(j)\})\cup\{\sigma(i)\}\in \m F.$ 
Moreover, $\sigma(T')=(S\setminus\{j\})\cup\{i\}\in \m F'.$
This proves Claim \ref{ijstable}.
\end{proof}
 
We claim that
\begin{equation}\label{6-15-2}
\text{$\m B'(1\bar2)\ne \emptyset$, $\mathcal A'(\bar1 \bar 2)\ne \emptyset$, and $\exists\, i\in B_0$ s.t. $\{2,i\}$ is non-full in $\m A'(\bar1)$.}    
\end{equation}

Indeed, 
since $\tau(\m B'_{\Delta})=\tau(\m B'(1))\ge 2$, $\m B'(1\bar2)\ne \emptyset$. 
If $\mathcal A'(\bar1 \bar 2)= \emptyset$, then by the maximality of $|\m A'||\m B'|$ (note that $|\m A'||\m B'|=|\m A||\m B|$), we see that all $k$-sets $B$ with $\{1,2\}\subset B\subset [n]$ belong to $\m B'$, together with $\m B'(1\bar2)\ne \emptyset$, forces $|\m B'(1)|>{n-2\choose k-2}$, contradicting $|\m B'(1)|\le{n-2\choose k-2}$. If for all $i\in B_0$, $\{2,i\}$ are full in $\m A'(\bar1)$, then by $n\ge 2k$, every set in $\m B'(1\bar2)$ (note that $\m B'(1\bar2)\ne \emptyset$) must contain $B_0$, this is impossible since $|\{1\}\cup B_0|>k$. This proves (\ref{6-15-2}).

\begin{claim}\label{6-17-4}
Let $A\in \m A'(\bar1\bar2)$. Then $|A\cap B_0|\ge 2$.
\end{claim}

\begin{proof}
Note that $|A\cap B_0|\ge 1$ by the cross-intersecting property.
Suppose,
for contradiction, that 
$A\cap B_0=\{t\}$ for some $t\in B_0=[3,k+2]$.
For the case $n=2k$, we have $|[n]\setminus(\{1,2\}\cup B_0)|=k-2$, which forces $|A\cap B_0|\ge 2$. 
 We may therefore assume
$n\ge 2k+1$.
Put $A_t:=\{t\}\cup [k+3,2k+1].$
By Claim \ref{ijstable}, $A_t\in \m A'(\bar1\bar2).$

We first show that every member of $\m B'(1\bar2)$ contains $t$.
Suppose not. Let $T\in\m B'(1\bar2)$ with $t\notin T$. Then
$|T|=k-1$ and $\{1\}\cup T\in\m B'$. Since $|[3,2k+1]\setminus (T\cup\{t\})|\ge k-1,$
we may choose a $(k-1)$-set $P\subset [3,2k+1]\setminus (T\cup\{t\}).$
Put $P_t:=\{t\}\cup P.$ 
Again, $P_t\in\m A'(\bar1\bar2)$ by
Claim \ref{ijstable}. But $P_t\cap(\{1\}\cup T)=\emptyset,$
contradicting the cross-intersecting property of $\m A'$ and $\m B'$.
Therefore every member of $\m B'(1\bar2)$ contains $t$.

Similarly, every member of $\m B'(12)$ contains $t$. Indeed, if
$T\in\m B'(12)$ and $t\notin T$, then $|T|=k-2$ and
$\{1,2\}\cup T\in\m B'$. Since $|[3,2k+1]\setminus (T\cup\{t\})|\ge k-1,$
we may choose a $(k-1)$-set $P\subset [3,2k+1]\setminus (T\cup\{t\}).$
As above, $P_t:=\{t\}\cup P$ belongs to $\m A'(\bar1\bar2)$, while 
$P_t\cap(\{1,2\}\cup T)=\emptyset,$
again contradicting cross-intersection.

Thus every member of $\m B'[1]$ contains $t$. Moreover,
$\m B'(\bar1)=\{B_0\}$ and $t\in B_0$. Hence every member of $\m B'$
contains $t$, so $\m B'$ is a star, contradicting the non-triviality of
$\m B'$. This proves $|A\cap B_0|\ge 2$.
\end{proof}

We already know that any one of the following three conditions forces $\tau(\m B'(1))=1$:
$\m A'(\bar1\bar2)=\emptyset$; $\m B'(1\bar2)=\emptyset$; or $\{2,i\}$ is full in
$\m A'$ for every $i\in B_0$. Combined with (i),  to prove~\eqref{gamma1-A0},
it suffices to establish one of these three conditions.
Note that $\m A'(\bar1\bar2)\ne\emptyset$, $\m B'(1\bar2)\ne\emptyset$,
and $\{2,i\}$ is non-full in $\m A'$ for some $i\in B_0$. 
We choose the minimum $i\in B_0$ that satisfies $\{2,i\}$ is non-full in $\m A'(\bar1)$. 
Starting from this situation,
we do the following three shifts repeatedly, stopping as soon as at least one of the three conditions above is met.

\begin{itemize}
\item[Step 1.]
If there exists $j\in B_0$ such that $(S_{2,j}(\m A'(\bar1)), S_{2,j}(\m B'[1]))\ne (\m A'(\bar1), \m B'[1])$, then we do this $S_{2,j}$-shift on $\m A'(\bar1)$ and $\m B'(1)$, and keep $B_0$.  

Put $\m R_1:=\m A'[1]\cup S_{2,j}(\m A'[\bar1])$ and $\m T_1:=S_{2,j}(\m B'[1])\cup \{B_0\}$.
We show that $\m R_1 $ and $ \m T_1$ are non-trivial cross-intersecting. Indeed, by Proposition \ref{1-21-2}, if $\m T_1$ is a star, then the center must be $2$; but $2\notin B_0\in \m T_1$, so 
$\m T_1$ is not a star; similarly, if $\m R_1$ is a star, then again the center is $2$, but $\m A'[1]\subset \m R_1$, (\ref{6-17-3}) also holds for $\m A'$, and $B_0=[3,k+2]$, $2$ is not the center of $\m R_1$, thus $\m R_1$ is not a star. For the cross-intersecting property, we take arbitrary $R_1\in \m R_1$ and $T_1\in \m T_1$, the only non-trivial case we need to check is $R_1\in S_{2,j}(\m A'[\bar1])$ and $T_1=B_0$. Claim \ref{6-17-4} gives $|R_1\cap B_0|\ge 1$. This proves $\m R_1 $ and $ \m T_1$ are non-trivial cross-intersecting. 

We shall also use the following simple fact. If $\m F\subset { [2,n]\choose k}$
is $2$-shifted and $2\in U$, then $S_{U,V}(\m F)$ is also $2$-shifted.
We repeatedly do such $S_{2,j}$-shifts for all possible $j\in B_0$.
After finishing Step 1, together with Claim \ref{ijstable}, both $\m A'(\bar1)$ and $\m B'(1)$ are $2$-shifted. 

\item[Step 2.] 
If there exists $A\in \m A'(\bar 1\bar 2)$ with $i\not\in A$, then apply Lemma \ref{spl}(i) with $X=[2,n]$,
$r=2$, $R=\{2,i\}$, there exists an $S_{U,V}^Q$-shift of $\m A'(\bar1)$ such that $R\subset U$. We do this $S_{U,V}^Q$-shift on $(\m A', \m B')$. 

We first explain why we can apply Lemma \ref{spl}(i). We check the required conditions of Lemma \ref{spl}(i) as follows: Since $\m A'(\bar1)$ is $2$-shifted, if $S_{U,V}(\m A'(\bar1))\ne \m A'(\bar1)$ for some $U$ with $2\in U$,
then $|V|\ge 2$, and hence (P1) holds; Since $A\in \m A'(\bar 1\bar 2)$ with $i\not\in A$, $A\cap \{2,i\}=\emptyset$, (P2) holds; Since $\{2,i\}$ is non-full in $\m A'(\bar1)$, (P3) holds.  Thus, we can apply 
Lemma \ref{spl}(i). The $S_{U,V}^Q$-shift of $\m A'(\bar1)$ with $i\in U$ is also of $(\m A', \m B')$. 

Put $\m R_2:=S_{U,V}^Q(\m A')$ and $\m T_2:=S_{U,V}^Q(\m B')$.
By Lemma \ref{SUV'}, $\m R_2$ and $\m T_2$ are cross-intersecting. 
Note that $\m A'(\bar1)$ and $\m B'(1)$ are $2$-shifted, and $2\in R\subset U$, by  Proposition \ref{suvcenter}, if $S_{U,V}^Q(\m A'(\bar1))$ or $S_{U,V}^Q(\m B'(1))$ is a star, then $2$ is a center. 
Since $i\in U\cap B_0$, $B_0$ is stable. An argument analogous to that in Step 1 shows that (\ref{6-17-3}) also holds for $S_{U,V}^Q(\m A')$, together with $B_0\in S_{U,V}^Q(\m B')$, neither $\m R_2$ nor $\m T_2$ is a star. Thus, $\m R_2$ and $\m T_2$ are non-trivial cross-intersecting.  

We repeatedly do such $S_{U,V}^Q$-shifts for all possible $(U,V)$ with $\{2,i\}\subset U$.
After finishing Step 2, we turn to Step 3. 
Now we are in the situation: $\{2,i\}$ is non-full, and every set in $\m A'(\bar 1\bar 2)$ contains $i$. Clearly, $\m A'(\bar 1\bar 2)\ne \emptyset$ since otherwise, we stop here. 

\item[Step 3.] 
If every set in $\m A'(\bar 1\bar 2)$ contains $i$, then apply Lemma \ref{spl}(ii) with $r=2$, $R=\{2\}$, there exists an $S_{U,V}^Q$-shift of $\m A'(\bar1)$ such that $R\subset U$ and $i\not\in V$. We do this $S_{U,V}^Q$-shift on $\m A'(\bar1)$ and $ \m B'(1)$, and keep $B_0$. 

We first explain why we can apply Lemma \ref{spl}(ii). We check the required conditions of Lemma \ref{spl}(ii) as follows: Since $R=\{2\}$, (P1) holds; Since $\m A'(\bar 1\bar 2)\ne \emptyset$, (P2) holds; Since $\{2,i\}$ is non-full in $\m A'(\bar1)$, both (P3) and (P4) hold; (P5) holds since every set in $\m A'(\bar 1\bar 2)$ contains $i$.
Thus, we can apply 
Lemma \ref{spl}(ii). 

Put $\m R_3:=\m A'[1]\cup S_{U,V}^Q(\m A'[\bar1])$ and $\m T_3:=S_{U,V}^Q(\m B'[1])\cup \{B_0\}$.
We show that $\m R_3 $ and $ \m T_3$ are non-trivial cross-intersecting. 
By Lemma \ref{SUV'}, $S_{U,V}^Q(\m A'[\bar1])$ and $S_{U,V}^Q(\m B'[1])$ are cross-intersecting. 
Note that every set in $\m A'(\bar 1\bar 2)$ contains $i$, $i\in B_0$ and $i\notin V$. This implies that every set in $S_{U,V}^Q(\m A'[\bar1])$ intersects $B_0$.
Clearly, $\m A'[1]$ and $\m T_3$ are cross-intersecting. We conclude that $\m R_3$ and $\m T_3$ are cross-intersecting.
Note that $2\in R\subset U$, and $\m A'(\bar1)$ and $\m B'(1)$ are $2$-shifted. By  Proposition \ref{suvcenter}, if $S_{U,V}^Q(\m A'(\bar1))$ or $S_{U,V}^Q(\m B'(1))$ is a star, then $2$ is the center. Thus, $\m T_3$ is not a star since $2\notin B_0\in \m T_3$; and $\m R_3$ is not a star since $\m A'[1]\subset \m R_3$, (\ref{6-17-3}) also holds for $\m A'$, and $B_0=[3,k+2]$, $2$ is not the center of $\m R_3$. This proves $\m R_3$ and $ \m T_3$ are non-trivial cross-intersecting. 
\end{itemize}

At termination, at least one of the following holds: $\mathcal A'(\bar1\bar2)=\emptyset,$ $\mathcal B'(1\bar2)=\emptyset,$ or $\{2,i\}$ is full in $\m A'$ for every $i\in B_0$. 
By the preceding  argument, $\tau(\m B'_{\Delta})$ decreases to $1$. 
This completes the proof.
\end{proof}

\subsection{Proof of Proposition \ref{cover2-nonshifted}}

\noindent\textit{Strategy of the proof.}\,
We argue by contradiction. In the present case, $\mathcal B$ is not
shifted, while Assumption~\ref{asm} implies that every
$S_{U,V}^{Q}$-shift of $\mathcal B$ is a star. Together with
$\tau(\mathcal B_\Delta)\ge2$, this imposes strong structural
restrictions on $\mathcal B$. On the other hand, the Frankl--Wang
bound~\eqref{minAB} shows that any counterexample must satisfy 
$|\mathcal B|>
\binom{n-3}{k-3}+\binom{n-4}{k-3}.$
Since the available estimate
$|\mathcal B(1)|\le\binom{n-2}{k-2}$ gives little control over
$\mathcal B(\bar1)$, we first determine the structure of $\mathcal B$
in four steps. 

\begin{enumerate}[(1)]
\item Set $M:=\cap_{B\in\mathcal B(\bar1)}B.$ 
We prove that, after fixing the intersection with $\{1\}\cup M$, the
remaining parts of the members of $\mathcal B$ form an L-initial family
on $[2,n]\setminus M$; see Claim~\ref{B-sections-lex}.

\item We choose elements $a,b,c,d,e$ so that
$S_{1,a}(\mathcal B)$, $S_{b,c}(\mathcal B)$, and
$S_{d,e}(\mathcal B)$ are stars, subject to the appropriate minimality
conditions. These choices yield properties
$(\mathrm P_{bc})$, $(\mathrm P_{de})$, and $(\mathrm P_M)$, and give
the decomposition $\mathcal B(1)=\mathcal B_1\mathbin{\cup}
\mathcal B_2\mathbin{\cup}\mathcal B_3,$
where $\mathcal B_1=\mathcal B(\{1,b\}\overline{\{c\}}),$ $\mathcal B_2=\mathcal B(\{1,c\}\overline{\{b\}}),$ and $\mathcal B_3=\mathcal B(\{1,b,c\}).$
This step also excludes the case $\gamma(\mathcal B)=1$.

\item Let $M_i:=\cap_{B\in\mathcal B_i}B,$ $i=1,2.$
Claims~\ref{2-7-2} and~\ref{2-7-5} show that $M_1\cap\{d,e\}=\{x\}$ and $M_2\cap\{d,e\}=\{y\}$, where $x\ne y$.

\item We finally determine these parameters. Claims~\ref{2-10-2},
\ref{dnea}, and~\ref{cnotinm} imply $M=\{a,b,d\}$, $d\ne a$, and $c\notin M$, 
and the preceding structure then forces $x=d$ and $y=e$.
\end{enumerate}

With this description of $\mathcal B$, the lower
bound~\eqref{minAB} yields a contradiction, except for a possible
configuration with $k=4$. In that case, $(\mathcal A,\mathcal B)\cong(\mathcal D_1,\mathcal D_2),$
but a direct calculation gives
$|\mathcal D_1||\mathcal D_2|<h(n,4)^2$, again contradicting
extremality.

\begin{proof}[Proof of Proposition \ref{cover2-nonshifted}]

We shall use the following consequence of Assumption \ref{asm}: 
for any $S_{i,j}$-shift, if $S_{i,j}(\m B)\ne \m B$, then
$S_{i,j}(\m B)$ is a star. Hence, if $S_{i,j}(\m B)$ is not a star, then
$S_{i,j}(\m B)=\m B$.
By (\ref{degree-normalization}), the following holds.
\begin{equation}\label{M-degree-downward}
\text{If } 2\le i<j\le n \text{ and } j\in M,\text{ then } i\in M.
\end{equation}

\begin{claim}\label{1-30-1}
Let $B\in \m B(\bar1)$. For any $i\in B\setminus M$, we have  $S_{1,i}(\m B)=\m B$ and $(B\setminus \{i\})\cup \{1\}\in \m B$.
\end{claim}
 
\begin{proof}
Let $i\in B\setminus M$. 
If $S_{1,i}(\m B)\ne \m B$, then $S_{1,i}(\m B)$ is a star with center 1. 
As $i\not\in M$, there exists $B'\in \m B(\bar{1})$ with $i\not\in B'$. Thus, $S_{1,i}(B')=B' \in S_{1,i}(\m B)$ and $1\not\in B'$,  a contradiction. 
Since $S_{1,i}(\m B)=\m B$, $(B\setminus \{i\})\cup \{1\}\in \m B$.
\end{proof}

\begin{claim}\label{1-29-4}
Consider an $S^{Q}_{U,V}$-shift of $\m B$ such that  $S^{Q}_{U,V}(\m B)$ is a star with center $c$. Then  either $c\in  M$ or $V\subset M$. In particular, if $S_{i,j}(\m B)$ is a star for some $1\le i<j \le n$, then $\{i,j\}\cap M\ne \emptyset$. 
\end{claim}
 
\begin{proof}
It is enough to prove the first assertion, since
by Proposition \ref{1-21-2}, $i$ is the center of $S_{i,j}(\m B)$, and hence, 
the second assertion follows by taking $U=\{i\}=\{c\}$ and $V=\{j\}$.

Since $\m B$ is non-trivial and $S_{U,V}(\m B)$ is a star with center $c$, by Proposition \ref{1-21-2}, we get $c\in U$, and every set in $\m B(\bar{c})$ must contain $V$. 
If $c=1$, then $V\subset M$, and we are done. Assume $c>1$. Suppose, for contradiction, that $c\not \in M$ and $V\not\subset M$. 
Then $\m B(\bar1 \bar c)\ne \emptyset$, and every set in $\m B(\bar1 \bar c)$ contains $V$. Let $B\in \m B(\bar1 \bar c)$. Then $V\subset B$. 
Since $V\not\subset M$, there exists $r\in V\setminus M$. Put $B':=(B\setminus \{r\}) \cup \{1\}$.
By Claim \ref{1-30-1}, $B'\in \m B$.   
Since $c\not\in B$, $c\not\in B'\in \m B(\bar c)$. 
However, $r\not\in B'$ implies $V\not\subset B'$, contradicting the fact that every set in $\m B(\bar{c})$ must contain $V$. 
\end{proof}

Let
\[
Y:=[2,n]\setminus M.
\]
For every
$C\in \{M\}\cup\{\{1\}\cup D:D\subset M\},$ denote
\[
\m B_C:=\{B\setminus C: B\cap (\{1\}\cup M)=C, B\in \m B\}.
\]

\begin{claim}\label{no-external-Q-shift}
There is no $S^Q_{U,V}$-shift of $\mathcal B$ with
$U,V\subset Y$.
\end{claim}

\begin{proof}
Suppose that such an $S^Q_{U,V}$-shift exists. Since it is also an
$S^Q_{U,V}$-shift of the pair $(\mathcal A,\mathcal B)$, Assumption
$\ref{asm}$ implies that $S^Q_{U,V}(\mathcal B)$ is a star. Let $c$
be its center. By Proposition \ref{1-21-2}, we have $c\in U$. By
Claim \ref{1-29-4}, either $c\in M$ or $V\subset M$.
If $c\in M$, then $c\in U\cap M$, contradicting $U\subset Y$. If
$V\subset M$, then $V\subset M\cap Y=\emptyset$, contradicting the
fact that the shift is non-trivial.
\end{proof}

The following claim shows that members of $\m B$ sharing the same intersection with $\{1\}\cup M$ have their remaining parts L-initial on $[2,n]\setminus M$.
\begin{claim}\label{B-sections-lex}
For every
$C\in \{M\}\cup\{\{1\}\cup D:D\subset M\},$
we have
$\mathcal B_C=\mathcal L(Y,|\mathcal B_C|,k-|C|).$
\end{claim}

\begin{proof}
Suppose, for contradiction, that there is $C\in \{M\}\cup\{\{1\}\cup D:D\subset M\}$ that violates Claim \ref{B-sections-lex}. 
Note that $\m B_C \subset {Y\choose k-|C|}$. 
Then $\mathcal B_C$ is not L-initial on $Y$. Then there is an $S_{U,V}^Q$-shift (acts non-trivially) of $\mathcal B_C$.
Since $U,V\subset Y$, this shift fixes every element of $C$. Hence the 
same shift is also an $S_{U,V}^Q$-shift of $\mathcal B$. 
This contradicts Claim $\ref{no-external-Q-shift}$. Therefore every section $\mathcal B_C$ is lexicographic, as claimed.
\end{proof}

We show that $\m B$ is not shifted. Assume for contradiction that $\m B$ is shifted. 
Since $(\mathcal A,\mathcal B)$ is not shifted, $\m A$ is not shifted, and there exists an $S_{i,j}$-shift such that $S_{i,j}(\m A) \ne \m A$, and hence this $S_{i,j}$-shift is also an $S_{\{i\},\{j\}}^Q$-shift of $(\m A, \m B)$.  
As any $S_{U,V}^Q$-shift of $(\m A, \m B)$ makes $\m B$ a star, $S_{i,j}(\m B)$ is a star.
But $S_{i,j}(\m B)=\m B$ and $\m B$ is non-trivial, a contradiction. 

We show that there exists $a$ such that $S_{1,a}(\m B)$ is a star.
Assume, for contradiction, that for all $i\in [2,n]$, $S_{1,i}(\m B)=\m B$. 
Since $\m B$ is not shifted, in view of Assumption \ref{asm}, there exists a pair $\{i,j\}\subset [2,n]$ such that $S_{i,j}(\m B)$ is a star, and both $S_{1,i}$, $S_{1,j}$ act trivially on $\m B$. By Lemma \ref{12-22-1}, every set in $\m B$ intersects $\{i,j\}$, and $\m B(i)\cap \m B(j)= \emptyset$. 
Let $B\in \m B(\bar{1})$. 
We claim that $B$ contains exactly one element of $\{i,j\}$, since otherwise both $(B\setminus \{i\})\cup \{1\}$ and $(B\setminus \{j\})\cup \{1\}$ belong to $\m B$, and hence $\m B(i)\cap \m B(j)\ne \emptyset$,
a contradiction. 
We may w.l.o.g. assume $B\cap \{i,j\}=\{i\}$. Let  $B':=(B\setminus \{i\}) \cup \{1\}$. Then $B' \in \m B$. However, $B'\cap \{i,j\}=\emptyset$, a contradiction. Thus, there exists $a$ such that $S_{1,a}(\m B)$ is a star. 
We choose $a$ to be the minimal element such that $S_{1,a}(\m B)$ is a star.

Next, we show that there exist $b, c \in [n]\setminus \{1,a\}$ such that $S_{b,c}(\m B)$ is a star. 
Let $T$ be a set consisting of the smallest $k-1$ elements in $[n]\setminus \{1,a\}$. 
Since $S_{1,a}(\m B)$ is a star, $\m B(1)\cap \m B(a)=\emptyset$ (by Lemma \ref{12-22-1}).
If $S_{b,c}(\m B)=\m B$ for all $b,c \in [n]\setminus \{1,a\}$ with $b<c$, then $T\in \m B(1)\cap \m B(a)$, a contradiction. Thus, there exist $b, c \in [n]\setminus \{1,a\}$ such that $S_{b,c}(\m B)$ is a star. 
We choose $(b,c)$ to be the (lexicographically) minimal pair such that $\{b,c\}\cap \{1,a\}=\emptyset$ and $S_{b,c}(\m B)$ is a star.
Applying Lemma \ref{12-22-1} to $S_{b,c}(\mathcal B)$, and using $\tau(\mathcal B_\Delta)\ge2$, we conclude that
\begin{itemize}
    \item[(P$_{bc}$)] Every set in $\m B(1)$ intersects $\{b,c\}$; 
both $\m B(\{1,b\}\overline{\{c\}})\ne \emptyset$, $ \m B(\{1,c\}\overline{\{b\}})\ne \emptyset$, and 
$\m B(\{1,b\})\cap \m B(\{1,c\})=\emptyset$.
\end{itemize}

We show that there exists a pair $\{d,e\}\subset [2,n]\setminus \{b,c\}$ such that $S_{d,e}(\m B)$ is a star. Let $R$ be a set consisting of the smallest $k-2$ elements in $[2,n]\setminus \{b,c\}$.
Suppose, for a contradiction, that every $S_{i,j}$-shift with $\{i,j\}\subset [2,n]\setminus \{b,c\}$ with $i<j$ acts trivially on $\m B$. Then in view of (P$_{bc}$), $R\in \m B(1,b)\cap \m B(1,c)$. 
This implies $\{1\}\cup R\in \m B(b)\cap \m B(c)$, contradicting Lemma $\ref{12-22-1}$. 
Then, there exists a pair $\{d,e\}\subset [2,n]\setminus \{b,c\}$ such that $S_{d,e}(\m B)$ is a star.
If possible, we choose such a pair $\{d,e\}$ with
$\{d,e\}\subset [2,n]\setminus\{a,b,c\}$; otherwise, we choose $(d,e)$ to be the minimal pair such that $\{d,e\}\subset [2,n]\setminus\{b,c\}$. 
Applying Lemma \ref{12-22-1} to $S_{d,e}(\mathcal B)$, and using $\tau(\mathcal B_\Delta)\ge2$, yields
\begin{itemize}
    \item[(P$_{de}$)] Every set in $\m B(1)$ intersects $\{d,e\}$; 
both $\m B(\{1,d\}\overline{\{e\}})\ne \emptyset$, $\m B(\{1,e\}\overline{\{d\}})\ne \emptyset$, and 
$\m B(\{1,d\})\cap \m B(\{1,e\})=\emptyset$.
\end{itemize}

Write
$$\m B_1:=\m B(\{1,b\}\overline{\{c\}}),\,\, \m B_2:=\m B(\{1,c\}\overline{\{b\}}),\,\, \m B_3:=\m B(\{1,b,c\}).$$
By (P$_{bc}$) and $\tau(\m B(1))\ge 2$, we have $\m B_1\ne \emptyset$ and $\m B_2\ne \emptyset$.

By Claim $\ref{1-29-4}$, and the choices of $a,b,c,d,e$, we have 
\begin{itemize}
\item[(P$_{M}$)] $a\in M$, $\{b,c\}\cap M\ne \emptyset$, $\{d,e\}\cap M\ne \emptyset$.
\end{itemize}

Since $S_{b,c}(\m B)$ and $S_{d,e}(\m B)$ are stars,  (P$_{bc}$) and (P$_{de}$) yield
\begin{align}
\label{6-5-1}
&|\m B(\{1,b, d\}\overline{\{c,e\}})|+|\m B(\{1,b, e\}\overline{\{c,d\}})|, |\m B(\{1,c, d\}\overline{\{b,e\}})|+|\m B(\{1,c, e\}\overline{\{b,d\}})|\le {n-5\choose k-3},\\ \label{6-17-1}
&|\m B(\{1,d,c\}\overline{\{b,e\}})|+|\m B(\{1,d,b\}\overline{\{c,e\}})|,
|\m B(\{1,e,c\}\overline{\{b,d\}})|+|\m B(\{1,e,b\}\overline{\{c,d\}})|
\le {n-5\choose k-3},\\
\label{6-5-2}
&|\m B(\{1,b,c, d\}\overline{\{e\}})|+|\m B(\{1,b,c, e\}\overline{\{d\}})|, |\m B(\{1,d,e, b\}\overline{\{c\}})|+|\m B(\{1,d,e, c\}\overline{\{b\}})|\le {n-5\choose k-4}.
\end{align}
 
Hence,
\begin{align}\label{2-7-3}
&|\m B_1|=|\m B(\{1,b, d,e\}\overline{\{c\}})|+|\m B(\{1,b, e\}\overline{\{c,d\}})|+|\m B(\{1,b, d\}\overline{\{c,e\}})|\le {n-4\choose k-3},\\ \label{2-7-4}
&|\m B_2|=|\m B(\{1,c, d,e\}\overline{\{b\}})|+|\m B(\{1,c, d\}\overline{\{b,e\}})|+|\m B(\{1,c, e\}\overline{\{b,d\}})|\le {n-4\choose k-3},\\ \label{2-7-1}
& |\m B_3|=|\m B(\{1,b,c, d,e\})|+|\m B(\{1,b,c, d\}\overline{\{e\}})|+|\m B(\{1,b,c, e\}\overline{\{d\}})|\le {n-4\choose k-4}.
\end{align}

Furthermore, combining (\ref{6-5-1}), (\ref{6-17-1}), and (\ref{6-5-2}), we have
\[
|\m B_1|+|\m B_2|\le 2{n-5\choose k-3}+{n-5\choose k-4}.
\]
Together with (\ref{2-7-1}), we have 
\begin{equation}\label{6-17-2}
|\m B(1)|=|\m B_1|+|\m B_2|+|\m B_3|\le {n-3\choose k-3}+{n-5\choose k-3}.
\end{equation}

We consider the case $\gamma(\m B)=1$. Then $|\m B(\bar1)|=1$.
Combining (\ref{6-17-2}) and (\ref{minAB}), $k=3$.
Write $\{d,e\}=\{x,y\}$, where we do not assume $x<y$.
Thus, $|\m B|=3$ and $\m B[1]=\{\{1,b,x\}, \{1,c,y\}\}$. 
We now estimate $|\m A|$: we have $|\m A(\bar1)|\le |\{A\in {[2,n]\choose k-3}: A\cap \{b,x\}\ne \emptyset, A\cap \{c,y\}\ne \emptyset\}|\le 4n-16$, and $|\m A(1)|\le 3n-9$ (note that every set in $\m A(1)$ intersects $M$), then $|\m A|\le 7n-25$. Thus $|\m A||\m B|\le 3(7n-25)<(3n-8)^2=h(n,3)^2$. A contradiction. 

Next, we may assume 
\[
\gamma(\m B)\ge 2.
\]

\begin{claim}\label{2-7-2}
    $\tau(\m B_1\cup \m B_2)= 2$.
\end{claim}
\begin{proof}
Since every set in $\m B_1\cup \m B_2$ intersects $\{d,e\}$ (by (P$_{de}$)), $\tau(\m B_1\cup \m B_2)\le 2$. 
Assume, for contradiction, that $\tau(\m B_1\cup \m B_2)=1$.
Let $x$ be a common element of all members of
$\mathcal B_1\cup\mathcal B_2$. 
By (P$_{bc}$), the families $\mathcal B_1$ and $\mathcal B_2$ are
disjoint. Since both are $(k-2)$-subfamilies of
${[n]\setminus\{1,b,c\}\choose k-2}$ and every member contains $x$, we
obtain $|\mathcal B_1|+|\mathcal B_2|
\le {n-4\choose k-3}.$
Together with \eqref{2-7-1}, this gives
$|\m B(1)|=|\m B_1|+|\m B_2|+|\m B_3|\le {n-3\choose k-3}.$
By (P$_M$), $b\in M$ or $c\in M$. 
If $b\in M$, then any set in $\m B(\bar{b})$ contains $\{1,c,x\}$, and hence $|\m B(\bar{b})|\le {n-4\choose k-3}$.
Similarly, if $c\in M$, then $|\m B(\bar{c})|\le {n-4\choose k-3}$. This implies that 
$|\m B(\bar{1})|=\gamma(\m B)=\min \{|\m B(\bar i)|: i\in [n]\}\le {n-4\choose k-3}.$
Thus, 
$|\m B|=|\m B(1)|+|\m B(\bar{1})|\le {n-3\choose k-3}+{n-4\choose k-3},$ 
contradicting (\ref{minAB}). 
\end{proof}

Write
$$M_1:=\cap_{B\in \m B_1}B, \,\, M_2:=\cap_{B\in \m B_2}B. $$
In what follows, we write $\{x,y\}$ for the pair $\{d,e\}$ when the order is irrelevant; that is, $\{x,y\}=\{d,e\}$ and we do not assume $x<y$, whereas $d<e$ throughout. 

\begin{claim}\label{2-7-5}
   For each $i\in \{1,2\}$, we have $|M_i\cap \{d,e\}|= 1$.
\end{claim}

\begin{proof}
We first prove that $|M_i\cap \{d,e\}|\le 1$ holds for each $i\in \{1,2\}$. We prove only for $i=1$, since the proof for $i=2$ is analogous. 
Assume, for contradiction, that $\{d,e\}\subset M_1$. Then $|\m B_1|=|\m B(\{1,d,e, b\}\overline{\{c\}})|$. 
By (\ref{6-5-2}),
\[
|\m B_1|+|\m B(\{1,d,e, c\}\overline{\{b\}})|=|\m B(\{1,d,e, b\}\overline{\{c\}})|+|\m B(\{1,d,e, c\}\overline{\{b\}})|\le {n-5\choose k-4}.
\]
Combining this with (\ref{6-5-1}) and 
\[
|\m B_2|=|\m B(\{1,d,e, c\}\overline{\{b\}})|+|\m B(\{1, c, d\}\overline{\{b,e\}})|+|\m B(\{1, c,e\}\overline{\{b,d\}})|,
\]
yields
\[
|\m B_1|+|\m B_2|\le {n-5\choose k-4}+{n-5\choose k-3}={n-4\choose k-3}.
\]
Together with (\ref{2-7-1}), this gives 
$
|\m B(1)|=|\m B_1|+|\m B_2|+|\m B_3|\le {n-4\choose k-3}+{n-4\choose k-4}={n-3\choose k-3}.
$
Note that $b\in M$ or $c\in M$. 
If $b\in M$, then by (\ref{2-7-4}), $|\m B(\bar{b})|=|\m B_2|\le {n-4\choose k-3}$;
if $c\in M$, then by (\ref{2-7-3}), $|\m B(\bar{c})|=|\m B_1|\le {n-4\choose k-3}$. 
Thus,
$\gamma(\m B)=|\m B(\bar{1})|\le {n-4\choose k-3}$, and hence $|\m B|=|\m B(1)|+|\m B(\bar{1})|\le {n-3\choose k-3}+{n-4\choose k-3},$ contradicting (\ref{minAB}). 
This proves $|M_i\cap \{d,e\}|\le 1$ holds for each $i\in \{1,2\}$.

We show that there exist two sets $B_1\in \m B_1, B_2\in \m B_2$ such that  $B_1\cap \{d,e\}=\{x\}$ and $B_2\cap \{d,e\}=\{y\}$. 
We use the facts that every set in $\m B_1\cup \m B_2$ intersects $\{d,e\}$ (by (P$_{de}$)), and $|M_i\cap \{d,e\}|\le 1$ for all $i\in \{1,2\}$.
For the case $|M_1\cap \{d,e\}|=0$, we may assume $M_2\cap \{d,e\}=\{y\}$, then there is $B_2$ with $B_2\cap \{d,e\}=\{y\}$, and $|M_1\cap \{d,e\}|=0$ gives $B_1$ with  $B_1\cap \{d,e\}=\{x\}$. The case $|M_2\cap \{d,e\}|=0$ is analogous. 
We now only need to consider the case $|M_1\cap \{d,e\}|=|M_2\cap \{d,e\}|=1$. 
Then Claim \ref{2-7-2} guarantees the existence of $B_1$ and $B_2$. 

It remains to prove that $x\in M_1$ and $y\in M_2$. 
Take an arbitrary $B\in \m B[1]$ with $|B\cap \{b,c\}|=1$. 
To prove Claim \ref{2-7-5}, it suffices to show that $\{b, x\}\subset B$ or $\{c,y\} \subset B$.

Put $R:=\{1,b\}\cup B_1$. 
Then $R\in \m B$, $c\not\in R$, $y\not\in R$ and $x\in R$. 
Let $R':=(R\setminus \{b\})\cup \{y\}$ and  $R'':=(R\setminus \{x\})\cup \{c\}$.
Put $T=\{1,c\}\cup B_2$.
Then $T \in \m B$, $b\not\in T$, $x\not\in T$ and $y\in T$. 
Let $T':=(T\setminus \{y\})\cup \{b\}$ and $T'':=(T\setminus \{c\})\cup \{x\}$. 
Then $R'\cap \{b,c\}=T''\cap \{b,c\}=\emptyset$ and $R''\cap \{x,y\}=T'\cap \{x,y\}=\emptyset$. 
Hence, $R', T''\not\in \m B$ by (P$_{bc}$), and $R'', T' \not\in \m B$ by (P$_{de}$). 

Consider the case $B\cap \{b,c\}=\{c\}$. We need to show that $y\in B$. Assume $y\not\in B$. Then $B\cap \{b,y\}=\emptyset$.  
By Lemma \ref{12-22-1}, neither $S_{y,b}(\m B)$ (if $y<b$) nor $S_{b, y}(\m B)$ (if $b<y$) is a star. This implies that  $S_{y,b}(\m B)=\m B$ (if $y<b$) or $S_{b, y}(\m B)=\m B$ (if $b<y$). 
We claim $y<b$, since otherwise, combining $y>b$ with $T'\not\in \m B$, we get $S_{b, y}(\m B)\ne\m B$, a contradiction.
Then $S_{y,b}(\m B)=\m B$. However, combining $y<b$ and $R'\not\in \m B$, we get $S_{y,b}(\m B)\ne \m B$, again a contradiction.

Consider the case $B\cap \{b,c\}=\{b\}$. We need to show that $x\in B$. Assume $x\not\in B$. 
Then $B\cap \{c, x\}=\emptyset$.
By Lemma \ref{12-22-1}, neither $S_{x,c}(\m B)$ (if $x<c$) nor $S_{c, x}(\m B)$ (if $c<x$) is a star. This implies that  $S_{x,c}(\m B)=\m B$ (if $x<c$) or $S_{c,x}(\m B)=\m B$ (if $c<x$).  
Since $R''\not\in \m B$, we have $x<c$, since otherwise, $x>c$ and $S_{c,x}(\m B)\ne \m  B$, contradicting $S_{c, x}(\m B)=\m B$. Then $S_{x,c}(\m B)=\m B$. However, combining $x<c$ and $T''\not\in \m B$, we get $S_{x,c}(\m B)\ne\m B$, again a contradiction.
This proves $x\in B$.
\end{proof}

Note that $\{d,e\}=\{x,y\}$ in some order.
By Claims \ref{2-7-2} and \ref{2-7-5}, $M_1\cap \{d,e\}$ and $M_2\cap \{d,e\}$ are distinct singletons; write $M_1\cap \{d,e\}=\{x\}$ and $M_2\cap \{d,e\}=\{y\}$. 

\begin{claim}\label{bcstar}
There is $i\in \{b,c\}$ such that $S_{1,i}(\m B)$ is a star and $i\in M$.
\end{claim}

\begin{proof}
By Claim $\ref{1-29-4}$, $\{b,c\}\cap M\ne \emptyset$. 
Consider $b\in M$.
If $c\not\in M$, then there exists $B\in \m B(\bar 1)$ such that $b\in B$ and $c\not\in B$, and hence $\{b,c\}\cap S_{1,b}(B)=\emptyset$, combining this with (P$_{bc}$), yields $S_{1,b}(B)\not\in \m B$, which implies that $S_{1,b}(\m B)$ is a star. 
If $c\in M$, and $S_{1,c}(\m B)$ is a star, then we are done.
Otherwise, $c\in M$, and $S_{1,c}(\m B)=\m B$. 
Let $B'\in \m B(\bar1)$. Then $B'':=(B'\setminus \{c\})\cup \{1\}\in \m B$.
In this case, if $S_{1,b}(\m B)=\m B$, then $B''':=(B'\setminus \{b\})\cup \{1\}\in \m B$. This implies $\m B(\{1,b\})\cap \m B(\{1,c\})\ne \emptyset$, contradicting (P$_{bc}$). This shows that if $b\in M$, then Claim \ref{bcstar} holds. Note that we never use the property $b<c$ in above analysis. Thus, by an analogous argument, we can show that the case $c\in M$ follows by interchanging $b$ and $c$.
\end{proof}

\begin{claim}\label{2-16-1}
Let $S\in \{\{c,x\}, \{b,y\}\}$, and let $i,j\in S$ and $i<j$. Then $S_{i,j}(\m B)$ is a star.
\end{claim}

\begin{proof}
Since $M_1\cap \{d,e\}=\{x\}$ and $M_2\cap \{d,e\}=\{y\}$, 
there exist $B_1\in \m B_1$ and $B_2\in \m B_2$ with $\{d,e\}\cap B_1=\{x\}$ and $\{d,e\}\cap B_2=\{y\}$. 
Put $T_1=\{1,b\}\cup B_1$ and $T_2=\{1,c\}\cup B_2$. 
Then $T_1, T_2 \in \m B$, $T_1\cap \{b,c,d,e\}=\{b,x\}$ and  $T_2\cap \{b,c,d,e\}=\{c,y\}$. 
Let $T'_1=(T_1\setminus \{b\})\cup \{y\}$, $T''_1=(T_1\setminus \{x\})\cup \{c\}$, $T'_2=(T_2\setminus \{c\})\cup \{x\}$, $T''_2=(T_2\setminus \{y\})\cup \{b\}$. 
By (P$_{bc}$) and (P$_{de}$), all $T'_1, T''_1, T'_2, T''_2$ are not in $\m B$. 
For the case $S=\{c,x\}$, if $c>x$, then since $T'_2\not\in \m B$, $S_{x,c}(\m B)\ne \m B$, and hence  $S_{x,c}(\m B)$ is a star; if $c<x$, then since $T''_1\not\in \m B$, $S_{c,x}(\m B)\ne \m B$, and hence  $S_{c,x}(\m B)$ is a star. The case $S=\{b,y\}$ follows by an analogous argument. 
\end{proof}

\begin{claim}\label{2-10-2}
Either $\{a,b,x\}\subset M$ or $\{a,c,y\}\subset M$.
\end{claim}

\begin{proof}
Note that $a\in M$, $M\cap \{b,c\}\ne \emptyset$, and $M\cap \{d,e\}\ne \emptyset$.
If $\{b,c\}\subset M$ or $\{d,e\}\subset M$, then we are done. 
Assume $M\cap \{b,c\}=\{b\}$, and suppose, for contradiction, that $x\not\in M$. Then there exists $B\in \m B(\bar{1})$ with $x\not\in B$. Define $B'$ as follows: if $c\in B$, then $B':=(B\setminus \{c\})\cup \{1\}$; if $c\not\in B$, then choose any element $i\in B\setminus M$ (note that $B\setminus M\ne \emptyset$, since $|\m B(\bar1)|=\gamma(\m B)\ge 2$), and let $B':=(B\setminus \{i\})\cup \{1\}$. Since $c\not\in M$ and $i\not\in M$, 
by Claim \ref{1-30-1}, we have $B'\in \m B$. Clearly, $x\not\in B'$ as $x\not\in B$. 
Thus, $\{1,b\}\subset B'$ (since $b\in M\subset B$) and $B'\cap \{c,x\}=\emptyset$. Then $x\not\in B'\setminus \{1,b\}\in \m B_1$. 
However, $M_1\cap \{d,e\}=\{x\}$ implies that every set in $\m B_1$ contains $x$, a contradiction. This proves that if $M\cap \{b,c\}=\{b\}$, then $\{a,b,x\}\subset M$. If $M\cap \{b,c\}=\{c\}$, then by an analogous argument we can show that $\{a,c,y\}\subset M$.
\end{proof}

By Claim \ref{2-10-2}, 
\begin{equation*}
\{a,b,d\}\subset M.
\end{equation*}

Indeed, if $\{a,b,x\}\subset M$, then either $x=d$, in which case
$\{a,b,d\}\subset M$ is immediate, or $x=e$, in which case $e\in M$ and
$d<e$, so $d\in M$ by \eqref{M-degree-downward}. Hence
$\{a,b,d\}\subset M$.
If $\{a,c,y\}\subset M$, then $c\in M$. Since $b<c$, we get
$b\in M$ by \eqref{M-degree-downward}. Similarly, we have $\{a,b,d\}\subset M$ in this case as well.

\begin{claim}\label{bcdestar} 
For each $P\in\{\{b,c\},\{d,e\}\}$, there exists
$p\in P\cap M$ such that $S_{1,p}(\mathcal B)$ is a star. 
\end{claim}

\begin{proof}
We prove the assertion for a fixed pair
$P=\{u,v\}\in\{\{b,c\},\{d,e\}\}$,
where $u<v$. By the choice of the pairs $(b,c)$ and $(d,e)$, the shift
$S_{u,v}(\mathcal B)$ is a star. Hence, by Lemma \ref{12-22-1}, every
member of $\mathcal B$ intersects $P$, and $\mathcal B(u)\cap\mathcal B(v)=\emptyset.$
Moreover, by (P$_M$), we have
$M\cap P\ne\emptyset$.

First suppose that $M\cap P=\{p\}$. Let $q$ be the other element of
$P$. Since $q\notin M$, there exists $B\in\mathcal B(\bar1)$ such that
$q\notin B$. As $p\in M$, we have $p\in B$. Put $B':=(B\setminus\{p\})\cup\{1\}.$
Then $B'\cap P=\emptyset$. Since every member of $\mathcal B$ intersects
$P$, we have $B'\notin\mathcal B$. Thus $S_{1,p}$-shift acts
non-trivially on $\mathcal B$, and hence $S_{1,p}(\mathcal B)$ is a star by Assumption \ref{asm}.

It remains to consider the case $M\cap P=P$. Take any
$B\in\mathcal B(\bar1)$. Then $u,v\in B$. We claim that at least one of
$S_{1,u}$-shift and $S_{1,v}$-shift acts non-trivially on $\mathcal B$. Indeed,
if both acted trivially, then $B_u:=(B\setminus\{u\})\cup\{1\}\in\mathcal B$ and $B_v:=(B\setminus\{v\})\cup\{1\}\in\mathcal B.$ 
But then $B_u\setminus\{v\}=B_v\setminus\{u\}=(B\setminus\{u,v\})\cup\{1\}\in \mathcal B(v)\cap\mathcal B(u),$
contradicting $\mathcal B(u)\cap\mathcal B(v)=\emptyset$. Hence at least one of
$S_{1,u}$ and $S_{1,v}$ acts non-trivially on $\mathcal B$. Let
$p\in\{u,v\}$ be such an element. Since $M\cap P=P$, we have $p\in M$.
Again by Assumption \ref{asm}, $S_{1,p}(\mathcal B)$ is a star.

Thus, in all cases, there exists $p\in P\cap M$ such that $S_{1,p}(\mathcal B)$ is a star. This proves the claim.
\end{proof}

\begin{claim}\label{dnea}
$d\ne a$.
\end{claim}

\begin{proof}
Suppose, for contradiction, that $d=a$. We first observe that 
\begin{equation}\label{b-shifted-2-n}
\text{$\m B$ is shifted on $[2,n]\setminus \{a,b,c\}$. }
\end{equation}
Indeed, if there were $i,j\in [2,n]\setminus \{a,b,c\}$ with $i<j$ such that
$S_{i,j}(\mathcal B)\ne\mathcal B$, then $S_{i,j}(\mathcal B)$ would be a star (by Assumption \ref{asm}), which would contradict the choice of $(d,e)$ (since otherwise, we would choose $(d,e)$ such that $a\not\in \{d,e\}$, contradicting $a=d$).

Consider the case $c\not\in M$. Since $M\cap \{b,c\}\ne \emptyset$ (by (P$_M$)), $b\in M$.
By Claim \ref{bcdestar}, $S_{1,b}(\m B)$ is a star. 
Let $R$ be the set of the $k-2$ smallest elements of $[2,n]\setminus \{a,b,c\}$. Then $\{a,b\}\cup R\in \m B(\bar1)$ by \eqref{b-shifted-2-n}. 
Since $\gamma(\m B)=|\m B(\bar1)|\ge 2$, $|M|\le k-1$.
Since $k\ge 3$, $R\ne\emptyset$. Thus, $R\setminus M\ne \emptyset$. Let $i:=\max (R\setminus M)$, and $R':=\{1,a,b\}\cup (R\setminus \{i\})$.
Consequently, $R\setminus\{i\}$ is the set of the $k-3$ smallest elements of $[2,n]\setminus \{a,b,c\}$.
By Claim \ref{1-30-1}, $R'\in \m B$. Note that $c\not\in R'$. So $R'\in \m B[\{1,a,b\}\overline{\{c\}}]$. Let $R'':=(R'\setminus \{b\}) \cup \{c\}$. Since $S_{b,c}(\m B)$ is a star, $\m B(b)\cap \m B(c)=\emptyset$. Thus, $R''\not\in \m B$.
We claim $\m B[\{1,a,c\}\overline{\{b\}}]=\emptyset$. Indeed, if not, then (\ref{b-shifted-2-n}) gives $R''\in \m B[\{1,a,c\}\overline{\{b\}}]$, contradicting $R''\not\in \m B$.
Note that $\m B(\{1,c\}\overline{\{b\}})=\m B_2\ne \emptyset$ and $M_2\cap \{d,e\}\ne \emptyset$. Thus, $e\in M_2$ (note that $a=d$) and 
$\m B[\{1,c,e\}\overline{\{b\}}]\ne \emptyset$. Moreover,  
\begin{equation}\label{6-12-1}
\m B_2=\m B(\{1,c,e\}\overline{\{a, b\}}),\,\,\m B[\{1,c,e\}\overline{\{a,b\}}]\ne \emptyset, \,\,    \m B[\{1,a,c,e\}\overline{\{b\}}]= \emptyset.  
\end{equation}

We claim $|\m B[\{1,c,e\}\overline{\{a,b\}}]|=1$. 
Suppose, for contradiction, that there exist two distinct sets $B, B'\in \m B[\{1,c,e\}\overline{\{a,b\}}]$. Then there is $j\in B\setminus B'$.
Let $B'':=(B'\setminus \{e\})\cup \{j\}$. Then $B''\cap \{b,c\}=\{c\}$ and $e\not\in B''$. Therefore $B''\not\in \m B$ by $e\in M_2$. This implies that if $j<e$, then (\ref{b-shifted-2-n}). Together with $j\not\in \{a,b,c\}$, yields $S_{j,e}(\m B)=\m B$, and hence  $B''\in \m B$, a contradiction. Thus $j>e$. Since $e>d=a$, $j>a$. 
Let $B''':=(B\setminus \{j\})\cup \{a\}$. Then $\{1,a,c,e\}\subset B'''$ and $b\not\in B'''$. 
By $B[\{1,a,c,e\}\overline{\{b\}}]= \emptyset$ (see (\ref{6-12-1})), $B'''\not\in \m B$. Thus $S_{a,j}(\m B)$ is a star (by the second condition of Assumption \ref{asm}). 
This implies that every set in $\m B$ intersects $\{a,j\}$ (by Lemma \ref{12-22-1}). However, $j\not\in B'$ and $a\not\in B'$, a contradiction. Thus, $|\m B[\{1,c,e\}\overline{\{a,b\}}]|=1$.
Note that $|\m B(\bar b)|=|\m B_2|$, moreover (\ref{6-12-1}) gives $|\m B_2|=|\m B(\{1,c,e\}\overline{\{a, b\}})|$. This implies that $\gamma(\m B)\le  |\m B(\bar b)|=|\m B_2|=1$, contradicting $\gamma(\m B)\ge 2$. This rules out the case $c\notin M$.

Consider the case $c\in M$. Then $\{a,b,c\}\subset M$. If $k=3$, then $|\m B(\bar1)|=1$, and hence $\gamma(\m B)=1$, contradicting $\gamma(\m B)\ge 2$. Thus $k\ge 4$.
Let $T$ be the set of the smallest $k-3$ elements of $[2,n]\setminus \{a,b,c\}$, and let $T':=\{a,b,c\}\cup T$. In view of (\ref{b-shifted-2-n}), $T'\in \m B$. 
By Claim \ref{bcdestar}, there is $p\in \{b,c\}$ such that $S_{1,p}(\m B)$ is a star. Let $\{q\}=\{b,c\}\setminus \{p\}$, and  $T'':=(T'\setminus \{p\})\cup \{1\}$. 
Since $S_{1,p}(\m B)$ is a star, by Lemma \ref{12-22-1}, $\m B(1)\cap \m B(p)=\emptyset$. Thus $T''\not\in \m B$. 
Note that $\{1,q,a\}\subset T''$, $p\not\in T''$, and $T''\setminus \{1,a,q\}$ is the set of the smallest $k-3$ elements of $[2,n]\setminus \{a,b,c\}$. Together with (\ref{b-shifted-2-n}),  we conclude that $\m B[\{1, a,q\}\overline{\{p\}}]=\emptyset$, since otherwise, $T''\in \m B$, a contradiction. 
If $p=b$, then $\m B[\{1,a,c\}\overline{\{b\}}]=\emptyset$ and $d=a$ implies (\ref{6-12-1}). Hence, by the same argument, we will get $\gamma(\m B)=1$, a contradiction.
If $p=c$, then $\m B[\{1,a,b\}\overline{\{c\}}]=\emptyset$ and $d=a$ implies
\begin{equation}\label{6-12-2}
\m B_1=\m B(\{1,b,e\}\overline{\{a, c\}}),\,\,\m B[\{1,b,e\}\overline{\{a,c\}}]\ne \emptyset, \,\,    \m B[\{1,a,b,e\}\overline{\{c\}}]= \emptyset.  
\end{equation}

Using the same argument as in the proof that $|\m B[\{1,c,e\}\overline{\{a,b\}}]|=1$, with the roles of $b$ and $c$ interchanged, we can get $|\m B[\{1,b,e\}\overline{\{a,c\}}]|=1$. Thus $|\m B_1|=1$, moreover (\ref{6-12-2}) gives $|\m B_1|=|\m B(\{1,b,e\}\overline{\{a, c\}})|$. This implies that $\gamma(\m B)\le  |\m B(\bar c)|=|\m B_1|=1$, contradicting $\gamma(\m B)\ge 2$. This rules out the case $c\in M$. Hence, $d\ne a$.
\end{proof}

Denote
\begin{align*}
\mathcal{F}_1 &:= \{F\in\mathcal{B}(1):\{b,c,x,y\}\subset F\}, &
\mathcal{F}_2 &:= \{F\in\mathcal{B}(1):\{b,x\}\subset F,\ \{c,y\}\cap F=\emptyset\}, \\[1pt]
\mathcal{F}_3 &:= \{F\in\mathcal{B}(1):\{b,c,x\}\subset F,\ y\notin F\}, &
\mathcal{F}_4 &:= \{F\in\mathcal{B}(1):\{b,x,y\}\subset F,\ c\notin F\}, \\[1pt]
\mathcal{F}_5 &:= \{F\in\mathcal{B}(1):\{b,c,y\}\subset F,\ x\notin F\}, &
\mathcal{F}_6 &:= \{F\in\mathcal{B}(1):\{c,x,y\}\subset F,\ b\notin F\}, \\[1pt]
\mathcal{F}_7 &:= \{F\in\mathcal{B}(1):\{c,y\}\subset F,\ \{b,x\}\cap F=\emptyset\}.
\end{align*}

By Claim \ref{2-16-1}, we have 
\begin{equation}\label{2-22-2}
|\m F_3|+|\m F_6|, \, |\m F_4|+|\m F_5| \le {n-5\choose k-4}.
\end{equation}

Note that every set in $\m B_3(\bar{x})$ contains $y$, and every set in $\m B_1(y)$ contains $x$; every set in $\m B_3(\bar{y})$ contains $x$, and every set in $\m B_2(x)$ contains $y$. We obtain
\begin{align}\label{6-11-1}
&|\m B_3(\bar{x})|=|\m B_3(\bar{x}y)|=|\m F_5|,\,\, |\m B_1(y)|=|\m B_1(xy)|=|\m F_4|,\\ \label{6-11-2}
&|\m B_3(\bar{y})|=|\m B_3(\bar{y}x)|=|\m F_3|,\,\, |\m B_2(x)|=|\m B_2(xy)|=|\m F_6|.
\end{align}

Combining (\ref{2-22-2}), (\ref{6-11-1}) and (\ref{6-11-2}), we have
\begin{align}\label{6-7}
&|\m B_3(\bar{x})|+|\m B_1(y)|=|\m F_5|+|\m F_4|\le {n-5\choose k-4},\\ \label{6-11-3}
&|\m B_3(\bar{y})|+|\m B_2(x)|=|\m F_3|+|\m F_6|\le {n-5\choose k-4}.
\end{align}

\begin{claim}\label{cnotinm}
 $c\not\in M$.
\end{claim}

\begin{proof}
Suppose, for contradiction, that $c\in M$. Then $\{a,b,c,d\}\subset M$. 
Since $S_{1,a}(\m B)$ is a star, $\m B(1)\cap \m B(a)=\emptyset$ (by Lemma \ref{12-22-1}), and hence 
\[
\{B\setminus \{1,b,c,d\}: B\in \m B, \{1,b,c,d\}\subset B\}\cap \{B\setminus \{a,b,c,d\}: B\in \m B(\bar1), \{a,b,c,d\}\subset B\}=\emptyset,
\]
which implies
\begin{equation}\label{6-6}
    |\m B(\bar{1})|+|\m B_3(d)|=|\m B(\bar{1})(\{a,b,c,d\})|+|\m B_3(d)| \le {n-4\choose k-4}.
\end{equation}

Note that $|\m B(1)|=|\m B_1|+|\m B_2|+|\m B_3|$, 
$|\m B_1(\bar y)|\le {n-5\choose k-3}$, $|\m B_2(\bar x)|\le {n-5\choose k-3}$, $|\m B_1|\le {n-4\choose k-3}$, and $|\m B_2|\le {n-4\choose k-3}$.
If $x=d$, then using (\ref{6-6})
together with (\ref{6-7}), this yields
\begin{align*}
    |\m B|&=|\m B(\bar{1})|+|\m B_3(d)|+|\m B_3(\bar{x})|+|\m B_1(y)|+|\m B_1(\bar{y})|+|\m B_2|\\
    &\le {n-4\choose k-4}+{n-5\choose k-4}+{n-5\choose k-3}+{n-4\choose k-3}\\
    &={n-3\choose k-3}+{n-4\choose k-3},
\end{align*}
contradicting (\ref{minAB}).
If $y=d$, then using (\ref{6-6}),
together with (\ref{6-11-3}), this yields 
\begin{align*}
    |\m B|&=|\m B(\bar{1})|+|\m B_3(d)|+|\m B_3(\bar{y})|+|\m B_2(x)|+|\m B_2(\bar{x})|+|\m B_1|\\
    &\le {n-4\choose k-4}+{n-5\choose k-4}+{n-5\choose k-3}+{n-4\choose k-3}\\
    &={n-3\choose k-3}+{n-4\choose k-3},
\end{align*}
contradicting (\ref{minAB}) again.
Thus $c\not\in M$. 
\end{proof}

Since $d\ne a$ (by Claim \ref{dnea}) and $\{a,b,d\}\subset M$, if $k= 3$, then $\m B(\bar1)=\{M\}$, and hence $\gamma(\m B)=1$, contradicting the assumption $\gamma(\m B)\ge 2$.
Then $k\ge 4$.
Put 
\begin{align*}
    &\m G_1:=\{G\in \m B(\bar{1}):c\in G, y\in G\},\,\,\,\,\,\,\,\,\,\,\,\,\,\,\,\,\,\,\m G_2:=\{G\in \m B(\bar{1}):c\in G, y\not\in G\},\\
&\m G_3:=\{G\in \m B(\bar{1}):c\not\in G, y\in G\},\,\,\,\,\,\,\,\,\,\,\,\,\,\,\,\,\,\,\m G_4:=\{G\in \m B(\bar{1}):c\not\in G, y\not\in G\}.
\end{align*}

By Claims \ref{bcdestar} and \ref{cnotinm}, $S_{1,b}(\m B)$ is a star. Then $a<b$ by the minimal choice of $a$. 

\begin{claim}
$e\ne a$. 
\end{claim}

\begin{proof}
Suppose, for contradiction, that $e=a$. Since $a<b$, 
$d<e=a<b<c$. 
Suppose $x=d$. Then $x=d<c$. Applying Claim \ref{2-16-1} with $S=\{c,x\}$, we see that $S_{d,c}(\m B)$ is a star. However, $a\ne d<b$, contradicting the minimal choice of $(b,c)$. 
Suppose $x=e$. Then $y=d$ and $y<b$. Applying Claim \ref{2-16-1} with $S=\{b,y\}$, we see that $S_{y,b}(\m B)$ is a star. However, $a\ne y<b<c$, contradicting the minimal choice of $(b,c)$. 
\end{proof}

We show that $x=d$ and $y=e.$ 
By Claim \ref{cnotinm}, $c\not\in M$. Since $d\in M$ and $M$ consists of the first $|M|$ elements of $[2,n]$, $d<c$. 
Suppose $x=e$. Then $y=d$. Applying Claim \ref{2-16-1} with $S=\{b,y\}$, we see that $S_{y,b}(\m B)$ is a star if $y=d<b$, or 
$S_{b,y}(\m B)$ is a star if $y=d>b$. However, $a\ne y=d<c$, contradicting the minimal choice of $(b,c)$. This proves $x\ne e$. Thus $x=d$ and $y=e$, as claimed.  

Since $e\ne a$ and $\m B(1)\cap \m B(a)=\emptyset$, $|\m F_1|+|\m G_1|\le {n-5\choose k-5}$ and $|\m F_2|+|\m G_4|\le {n-5\choose k-3}$.
Clearly, $|\m F_7|\le {n-5\choose k-3}$, $|\m G_2|+|\m G_3|\le 2 {n-6\choose k-4}$ (note that every member of $\mathcal G_2$ contains $\{a,b,c,d\}$, avoids
$\{1,e\}$, and every member of $\mathcal G_3$ contains
$\{a,b,d,e\}$, avoids $\{1,c\}$.).
By (\ref{2-22-2}) and (\ref{minAB}),  
\begin{equation}\label{2-20-1}
   {n-3\choose k-3}+{n-4\choose k-3}+1 \le |\m B|=\sum_{i\in [7]}|\m F_i|+\sum_{i\in [4]}|\m G_i|\le 2{n-4\choose k-3}+{n-5\choose k-5}+2{n-6\choose k-4}.
\end{equation}
Then
\begin{equation}\label{6-6-1}
2{n-4\choose k-3}+{n-5\choose k-5}+2{n-6\choose k-4}-|\m B|< {n-6\choose k-4}-{n-6\choose k-5}.
\end{equation}
The upper bound in \eqref{2-20-1} is obtained by adding the six estimates: $|\m F_1|+|\m G_1|\le {n-5\choose k-5}$, $|\m F_2|+|\m G_4|\le {n-5\choose k-3}$, $|\m F_7|\le {n-5\choose k-3}$, $|\m G_2|+|\m G_3|\le 2 {n-6\choose k-4}$, $|\m F_3|+|\m F_6|\le {n-5\choose k-4}$, and $|\m F_4|+|\m F_5| \le {n-5\choose k-4}$. 
Put
\[
D:={n-6\choose k-4}-{n-6\choose k-5}.
\]
By \eqref{6-6-1}, the total deficit from these six upper bounds is less
than $D$. Hence no single one of these six estimates can have deficit at
least $D$.
This implies the following claim. 
\begin{claim}\label{6-6-2}
(i) $\m G_2\ne \emptyset$, $\m G_3\ne \emptyset$, and $|\m G_2|+|\m G_3|> 2 {n-7\choose k-5}$;

(ii) $ |\m F_7|>{n-5\choose k-3}-\left({n-6\choose k-4}-{n-6\choose k-5}\right)={n-6\choose k-3}+{n-6\choose k-5}$.
\end{claim}

\begin{proof}
If $\mathcal G_2=\emptyset$ or $\mathcal G_3=\emptyset$, then the deficit
in the estimate $|\mathcal G_2|+|\mathcal G_3|\le 2{n-6\choose k-4}$ is at least ${n-6\choose k-4}\ge D$, a contradiction. Thus
$\mathcal G_2,\mathcal G_3\ne\emptyset$.
Similarly, if $|\mathcal G_2|+|\mathcal G_3|\le 2{n-7\choose k-5},$ then the same estimate has deficit at least $2{n-6\choose k-4}-2{n-7\choose k-5}>D,$
again a contradiction. We get (i).
If $|\mathcal F_7|\le {n-6\choose k-3}+{n-6\choose k-5},$
then the deficit in $|\mathcal F_7|\le {n-5\choose k-3}$
is at least ${n-5\choose k-3}-{n-6\choose k-3}-{n-6\choose k-5}
=
{n-6\choose k-4}-{n-6\choose k-5}
=D,$
contradicting \eqref{6-6-1}. This proves (ii).
\end{proof}

We claim $M=\{a,b,d\}$. Suppose not, and choose 
$z\in M\setminus\{a,b,d\}.$
Since $c\notin M$, we have $z\ne c$. If $z=e$, then every member of
$\mathcal B(\bar1)$ contains $e$, and hence $\mathcal G_2=\emptyset$,
contradicting Claim \ref{6-6-2}(i).
Thus $z\notin\{c,e\}$. Then every member of $\mathcal G_2\cup\mathcal G_3$
contains $z$ in addition to $\{a,b,d\}$ and one of $c,e$. Therefore $|\mathcal G_2|+|\mathcal G_3|\le 2{n-7\choose k-5}$,
contradicting Claim \ref{6-6-2}(i). Hence $M=\{a,b,d\}$.

Consequently, $M\cap\{b,c\}=\{b\}$ and $M\cap\{d,e\}=\{d\}$. By Claim \ref{bcdestar}, both $S_{1,b}(\m B)$ and  $S_{1,d}(\m B)$ are stars.
By the minimal choice of $a$, we get 
\[
a<b<c, \,\, a<d.
\]

\begin{claim}\label{2-27-3}
 If $k= 4$, then $(\m A, \m B)$ is isomorphic to $(\m D_1, \m D_2)$, and $|\m A||\m B|<h(n,k)^2$. 
\end{claim}

\begin{proof}
Substituting $k=4$ into inequality (\ref{2-20-1}) shows that
$|\mathcal{B}|$ exactly attains the upper bound.
This forces $|\m F_2|+|\m G_4|={n-5\choose k-3}$, $|\m F_7|={n-5\choose k-3}$, and $|\m F_3|+|\m F_6|=|\m F_4|+|\m F_5|= {n-5\choose k-4}$ (by (\ref{2-22-2})). 
Note that $M=\{a,b,d\}$ and $a\ne d$. 
Since $\m G_2\ne \emptyset$ and $\m G_3\ne \emptyset$ (by Claim \ref{6-6-2}(i)), $\{a,b,c,d\}$ and $\{a,b,d,e\}$ are in $\m B(\bar{1})$. 
By Claim \ref{1-30-1}, $\{1,a,b,d\}\in \m B$ since $c\not\in M$. 
Since $S_{1,a}(\m B)$ is a star, $\{1,b,c,d\}$ and $\{1,b,d,e\}$ are not in $\m B$, i.e. $\m F_3=\m F_4=\emptyset$. 
So $|\m F_5|+|\m F_6|+|\m F_7|=2{n-5\choose k-4}+{n-5\choose k-3}={n-3\choose k-3}$. Thus, 
\begin{equation}\label{1ce}
\text{All $4$-sets $B$ with $\{1,c,e\}\subset B\subset [n]$ belong to $\m B$. }
\end{equation}

We show that 
\begin{equation}\label{4-bar1-full}
\m B(\bar{1})=\{\{a,b,d\}\cup \{i\}: i\in [n]\setminus \{a,b,d\}\}.
\end{equation}

Suppose, for contradiction, that there exists a set $\{a,b,d,i\}\not\in \m B$ for some $i\in [n]\setminus \{a,b,d\}$. 
Since $\{1,a,b,d\}\in \m B$, $i\in [2,n]\setminus \{a,b,d\}$.
Recall that $M=\{a,b,d\}$ and $\{a,b,c,d\}\in\m B(\bar{1})$. 
Then $\m B(1)\cap \m B(a)=\emptyset$ (since $S_{1,a}(\m B)$ is a star) gives $\{1,b,c,d\}\not\in \m B$.
Since $\m B(\bar1)(M)$ is L-initial on $[2,n]\setminus M$ (by Claim \ref{B-sections-lex}), we have $i>c$. 
As $|\m F_2|+|\m G_4|={n-5\choose k-3}$, the assumption $\{a,b,d,i\}\not\in \m B(\bar{1})$ forces $\{1,b,d,i\}\in \m B$. 
Applying Claim \ref{B-sections-lex} with $D=\{b,d\}$, since $M=\{a,b,d\}$ and $i>c$, $\{1,b,d,i\}\in \m B$ implies $\{1,b,c,d\}\in \m B$, a contradiction. This proves (\ref{4-bar1-full}).

We claim $\m B[\{1,b,d\}\overline{\{a\}}]=\emptyset$. Indeed, for any $i\in [2,n]\setminus \{a,b,d\}$, by (\ref{4-bar1-full}), we have $\{a,b,d,i\}\in \m B$, and hence $\{1,b,d,i\}\not\in \m B$ (by $\m B(1)\cap \m B(a)=\emptyset$), so 
$\m B[\{1,b,d\}\overline{\{a\}}]=\emptyset$, as claimed. 
Thus, 
\[
\m B[1]=\{\{1,a,b,d\}\}\cup \m B[\{1,c,e\}].
\]
Combining this with (\ref{1ce}) and (\ref{4-bar1-full}),
we conclude that
\[
\m B=\left\{B\in {[n]\choose 4}: \{1,c,e\}\subset B\right\}\cup \left\{B\in {[n]\choose 4}: \{a,b,d\}\subset B\right\}.
\]
Write $\m C:=\{C: |C\cap \{a,b,d\}|=|C\cap \{1,c,e\}|=1, |C|=2 \}$.
Since $n\ge2k$, combining (\ref{1ce}) and (\ref{4-bar1-full}), the maximal choice of $(\m A, \m B)$ gives
$$\m A=\left\{A\in {[n]\choose 4}: \exists \, C\in \m C \,s.t.\, C \subset A\right\},$$
so $(\m A, \m B)\cong(\m D_1, \m D_2)$.
By (\ref{2-27-4}), we get $|\m A||\m B|<h(n,k)^2$.
\end{proof}

We may next assume $k\ge 5$. 

\begin{claim}\label{ce-min}
$\min\{c,e\}=\min([2,n]\setminus M).$
\end{claim}

\begin{proof}
Suppose, for contradiction, that there exists $t\in [2,n]\setminus M$ with $t<\min\{c,e\}.$
Then $t<c$ and $t<e$. In particular, 
$t\notin \{a,b,c,d,e\},$
since $M=\{a,b,d\}$ and $c,e\notin M$.

We first show that there exists $F\in\mathcal F_7$ such that $t\notin F$.
Indeed, if every member of $\mathcal F_7$ contains $t$, then, since
every member of $\mathcal F_7$ contains $\{c,e\}$ and avoids $\{b,d\}$,
we would have 
$|\mathcal F_7|\le {n-6\choose k-4},$ contradicting Claim \ref{6-6-2}(ii) (using $n\ge 2k$ and $k\ge5$).
 Thus we may choose
 $F\in\mathcal F_7$ with $t\notin F.$

Recall that $F\in\mathcal B(1)$, so $\{1\}\cup F\in\mathcal B$. Also,
by the definition of $\mathcal F_7$, we have
$\{c,e\}\subset F$ and $\{b,d\}\cap F=\emptyset.$
Since $t<c$, consider
$F_c:=\bigl((\{1\}\cup F)\setminus\{c\}\bigr)\cup\{t\}.$
Then
$F_c\cap\{b,c\}=\emptyset,$
and hence $F_c\notin\mathcal B$ by (P$_{bc}$). Therefore
$S_{t,c}$-shift acts non-trivially on $\mathcal B$. By Assumption \ref{asm},
$S_{t,c}(\mathcal B)$ is a star. Hence, by Lemma \ref{12-22-1}, 
\begin{equation}\label{eachBtc}
\text{every member of $\mathcal B$ intersects $\{t,c\}$.}  
\end{equation}
Similarly, since $t<e$, consider
$F_e:=\bigl((\{1\}\cup F)\setminus\{e\}\bigr)\cup\{t\}.$
Then
$F_e\cap\{d,e\}=\emptyset,$
and hence $F_e\notin\mathcal B$ by (P$_{de}$). Therefore
$S_{t,e}$-shift acts non-trivially on $\mathcal B$, and again by
Assumption \ref{asm}, $S_{t,e}(\mathcal B)$ is a star. Thus 
\begin{equation}\label{eachBte}
\text{every member of $\mathcal B$ intersects $\{t,e\}$.}
\end{equation}
Combining (\ref{eachBtc}) and (\ref{eachBte}), every member of both
$\mathcal G_2$ and $\mathcal G_3$ contains $t$, then
$|\mathcal G_2|\le {n-7\choose k-5}$ and $|\mathcal G_3|\le {n-7\choose k-5},$
and hence $|\mathcal G_2|+|\mathcal G_3|\le 2{n-7\choose k-5},$
contradicting Claim \ref{6-6-2}(i).
\end{proof}

Let $p=\min \{c,e\}$ and $q=\max \{c,e\}$.
Combining Claim \ref{ce-min} and $\m B(\bar 1)(M)$ is L-initial on $[2,n]\setminus M$ (by Claim \ref{B-sections-lex}), we observe that every set in $\m B(\bar1)$ containing $\{p,q\}$ precedes
any set in $\m B(\bar1)$ containing $q$ but not $p$ in the lex order. Therefore, by $\m G_2\ne \emptyset$ and $\m G_3\ne \emptyset$, we conclude that 
\begin{equation}\label{full-abdce}
\text{every $k$-set $B$ with $\{a,b,c,d,e\}\subset B\subset [2,n]$ belongs to $\mathcal B$.}
\end{equation}

Recall that $S_{1,a}(\m B)$, $S_{1,b}(\m B)$ and  $S_{1,d}(\m B)$ are stars. Together with (\ref{full-abdce}), we obtain
\begin{equation}\label{2-23-1}
\m B(\{1,a,b,c,e\},\overline{\{d\}})=\m B(\{1,a,c,d,e\},\overline{\{b\}})=\m B(\{1,b,c,d,e\},\overline{\{a\}})=\emptyset.
\end{equation}

We claim 
$|\{B\in \m B(1): B\cap \{a,b,c,d,e\}=\{a,c,e\}\}|\ge 2$, since otherwise, 
$|\m F_7|\le 1+ |\{B\in \m B(1): B\cap \{a,b,c,d,e\}=\{c,e\}\}|\le 1+{n-6\choose k-3}$, contradicting Claim \ref{6-6-2}(ii).
So there exist $B, B'\in \{B\in \m B(1): B\cap \{a,b,c,d,e\}=\{a,c,e\}\}$ and $i$ such that $i\in B\setminus B'$. Moreover, $b<i$ (by $b\in M$ and (\ref{6-9-1}).
Let $B''=((\{1\}\cup B)\setminus \{i\})\cup \{b\}$. Then $\{1,a,b,c,e\}\subset B''$ and $d\not\in B''$. In view of (\ref{2-23-1}), $B''\not\in \m B$, together with $b<i$, this gives $S_{b,i}(\m B)\ne \m B$, and hence   
$S_{b,i}(\m B)$ is a star by
Assumption \ref{asm}. Thus, every set in $\m B(1)$ intersects $\{b,i\}$ by Lemma \ref{12-22-1}.  However, $B'\cap \{b,i\}=\emptyset$, a contradiction. 
This completes the proof of Proposition \ref{cover2-nonshifted}. 
\end{proof}

\section{Characterization of extremal families: Proof of Theorem \ref{thm:equality}}\label{sec:equality}
In this section, we characterize all extremal families that attain the
extremal value $h(n,k)^2$ for $n\ge 2k$ and $k\ge 3$. 
The following result, by Frankl and Kupavskii \cite{FK-diversity}, is useful for handling $k\ge 4$. 
\begin{theorem}[Frankl--Kupavskii \cite{FK-diversity}]
\label{FK-cross-diversity}
Let $n\ge2k$, let $3\le u\le k$, and let
$\m A,\m B\subset\binom{[n]}k$ be cross-intersecting.
Suppose that
\[
|\m A|,|\m B|
>
\binom{n-1}{k-1}
-\binom{n-u-1}{k-1}
+\binom{n-u-1}{k-u}.
\]
Then
\[
\gamma(\m A),\gamma(\m B)
<
\binom{n-u-1}{k-u}.
\]
Moreover, $\m A$ and $\m B$ have the same unique element of
maximum degree.
\end{theorem}

For $k=3$, Frankl \cite{11-6-2} established the upper bound, without describing the extremal families; for completeness, we give all extremal families for
$k=3$ here. Since the result of Frankl and Kupavskii does not apply when $k=3$, this case needs more analysis. We prove the following result, whose proof is self-contained and does not rely on the main text, so we place it in the {\bf Appendix}.

\begin{proposition}\label{end3}
Let $n\ge 7$, $\m A, \m B \subset {[n]\choose 3}$ be non-trivial cross-intersecting families with $|\m A|=|\m B|=3n-8$. Then $(\m A, \m B)$ is of cross-HM type or cross-triangle type.
\end{proposition}

Now, we are ready to prove Theorem \ref{thm:equality}. 

\begin{proof}[Proof of Theorem \ref{thm:equality}]
A direct calculation shows that the three types of constructions in the statement are extremal. 
Let $\m A, \m B\subset {[n]\choose k}$ be non-trivial cross-intersecting families with $|\m A||\m B|=h(n,k)^2$. By symmetry, we may assume $|\m A|\ge h(n,k)$. 
We claim 
\begin{equation}\label{6-19-1}
|\m A|+|\m B|\le 2h(n,k).
\end{equation}
Indeed, we choose the lex-minimal non-trivial cross-intersecting pair $(\m A', \m B')$ such that $|\m A'|=|\m A|$ and $|\m B'|=|\m B|$. Then also $|\m A'||\m B'|=h(n,k)^2$. 
Since $|\m A'|=|\m A|\ge h(n,k)$, by Lemma \ref{mainlem0}, $\Delta(\m A')\ge {n-2\choose k-2}+1$. If $\Delta(\m B')\ge {n-2\choose k-2}+1$, then by Lemma \ref{mainlem1}, we have $|\m A'|+|\m B'|\le 2h(n,k),$ and hence $|\m A|+|\m B|=|\m A'|+|\m B'|\le 2h(n,k)$, as required. If $\Delta(\m B')\le {n-2\choose k-2}$, then by Corollary \ref{mainlem2-coro}, either $|\m A'||\m B'|< h(n,k)^2$ or  $|\m A'|+|\m B'|\le 2h(n,k)$. The former alternative is impossible since otherwise $|\m A||\m B|=|\m A'||\m B'|< h(n,k)^2$, a contradiction. The latter case implies (\ref{6-19-1}). This proves (\ref{6-19-1}).  

Combining $|\m A||\m B|=h(n,k)^2$ with (\ref{6-19-1}), we have
\begin{equation}\label{6-19-2}
|\m A|=|\m B|=h(n,k).
\end{equation}

We first consider the case $n=2k$. Let $\overline{\m A}:=\{[2k]\setminus A:A\in \m A\}$. 
 Since $\m A$ and $\m B$ are
cross-intersecting, $\m B\subset
\binom{[2k]}k\setminus\overline{\m A}$. 
Both sides have size $h(2k,k)$; hence $\m B=
\binom{[2k]}k\setminus\overline{\m A}
=\m A^\perp.$ Taking $\m A=\m C$ and $\m B=\m C^{\perp}$, we conclude that $(\m A, \m B)$ is of complement-dual type.

Next, we may assume $n>2k$. For $k=3$, Theorem \ref{thm:equality}  follows from Proposition \ref{end3}. Assume $k\ge 4$. 
By (\ref{6-19-2}), a direct
calculation gives
\[
|\m A|,\,|\m B|>\binom{n-1}{k-1}-\binom{n-k}{k-1}+\binom{n-k}{1}.
\]
Applying Theorem~\ref{FK-cross-diversity} with $u=k-1$, we
find that $\m A$ and $\m B$ have the same unique element of
maximum degree, say $1$, and $\gamma(\m A),\gamma(\m B)<n-k$.
Since $n>2k$ and $k\ge4$ give $n-k\le\binom{n-3}{k-2}$, we
have $\gamma(\m A),\gamma(\m B)\le\binom{n-3}{k-2}$, so the
proof of Lemma~\ref{1-21-1} applies, with $1$ as the common
maximum-degree element.

We claim $\gamma(\m A)=\gamma(\m B)=1$.
In the notation of the proof of Lemma~\ref{1-21-1}, the
equality $|\m A|+|\m B|=2h(n,k)$ forces
$|\m A_1|+|\m B_2|=2h(n,k)$ with $|\m A_1|=|\m A|=h(n,k)$,
hence $|\m B_2|=|\m B|=h(n,k)$; consequently both
non-trivial intersecting families $\m A'\cup\m B_2(\bar1)$
and $\m B'\cup\m A_1(\bar1)$ have size exactly $h(n,k)$.
Since $n>2k$ and $k\ge4$, the uniqueness part of the
Hilton--Milner theorem (Theorem~\ref{hm-nontrivial}) forces
each of them to be isomorphic to $\m{HM}(n,k)$, whose
diversity part is a single set. As $1$ is the center of
these two families, their diversity parts are
$\m B_2(\bar1)$ and $\m A_1(\bar1)$, respectively; hence
$|\m A_1(\bar1)|=|\m B_2(\bar1)|=1$. Note that
$|\m A_1(\bar1)|=|\m A(\bar1)|=\gamma(\m A)$ and
$|\m B_2(\bar1)|=|\m B(\bar1)|=\gamma(\m B)$. Therefore
$\gamma(\m A)=\gamma(\m B)=1$. 

Write $\m A(\bar1)=\{A_0\}$ and $\m B(\bar1)=\{B_0\}$; then
$1\notin A_0\cup B_0$ and $A_0\cap B_0\ne\emptyset$. Then
$\m A=\{A:1\in A,\ A\cap B_0\ne\emptyset\}\cup\{A_0\}
=\m H_1(A_0;B_0)$ and $\m B=\m H_1(B_0;A_0)$. Hence $(\m A,\m B)$ is
of cross-HM type, as required.
This completes the proof.
\end{proof}

\section{Concluding remarks}\label{sec:concluding}
It is natural to ask for the maximum of
$|\m A||\m B|$ when $\m A$ and $\m B$ have different
uniformities.
Let $n \ge 2a \ge 2b \ge 4$ be integers and fix
$A_0 \in \binom{[2,n]}{a}$, $B_0 \in \binom{[2,n]}{b}$ with
$A_0 \cap B_0 \ne \emptyset$. Define
\[
\mathcal{A}_0 = \{A_0\} \cup \left\{ A \in \binom{[n]}{a} :
1 \in A,\ A \cap B_0 \ne \emptyset \right\},
\]
\[
\mathcal{B}_0 = \{B_0\} \cup \left\{ B \in \binom{[n]}{b} :
1 \in B,\ B \cap A_0 \ne \emptyset \right\}.
\]
Frankl and Wang \cite{FW2026} proposed the following conjecture.

\begin{conjecture}[Frankl--Wang \cite{FW2026}]\label{FW-conj}
Let $n \ge 2a \ge 2b \ge 4$. Suppose that
$\mathcal{A} \subset \binom{[n]}{a}$,
$\mathcal{B} \subset \binom{[n]}{b}$ are non-trivial cross-intersecting families. Then
\[
|\mathcal{A}||\mathcal{B}| \le |\mathcal{A}_0||\mathcal{B}_0|.
\]
\end{conjecture}
We show that Conjecture \ref{FW-conj} fails in general.
Let us introduce a counterexample.

\begin{construction}\label{3-11-1}
    \begin{align*}
        &\m C_a=\left\{C\in {[n]\choose a}: \{1,2\}\subset C\right\} \cup \left\{[3,a+2]\right\},\\
        &\m C_b=\left\{C\in {[n]\choose b}: \{i,j\}\subset C, i\in [2],j\in [3,a+2]\right\}.
    \end{align*}
\end{construction}

Taking $a=b=k$ in the above construction, we have $\m C_1=\m C_a$ and $\m C_2=\m C_b$ (where $\m C_1, \m C_2$ are defined in Construction \ref{2-27-1}). By Proposition \ref{3-11-2}, if $a=b$, then $|\m C_a||\m C_b|<h(n,a)h(n,b)=|\m A_0||\m B_0|$. However, when $a>b$, for example, $a=9, b=3, n=18$, we have $|\m C_a||\m C_b|>|\m A_0||\m B_0|$.

\section*{Acknowledgments}
I am grateful to Andrey Kupavskii for bringing this problem to my attention.

\section*{Appendix}
In this section, we give the proof of Proposition \ref{end3}. We need the following two lemmas.

\begin{lemma}\label{end1}
Let $n\ge 7$, and let
$\m A,\m B\subset\binom{[n]}{3}$ be non-trivial
cross-intersecting families. If $|\m A|=|\m B|=3n-8$, then $\Delta(\m A)\ge n-1$ or $\Delta(\m B)\ge n-1$.
\end{lemma}

\begin{proof}
Suppose, for contradiction, that $\Delta(\m A)\le n-2$ and $\Delta(\m B)\le n-2$.
We may w.l.o.g. assume $\m B_{\Delta}=\m B(1)$. Then $|\m B(1)|\le n-2$, and hence, $|\m A(\bar1)|\ge(3n-8)-(n-2)=2n-6.$ 
If $n=7$, then $|\m B|=\frac{1}{3}\sum_{x\in[n]}d_{\m B}(x)\le \frac{1}{3}n\Delta(\m B)\le \frac{35}{3}<3n-8$, a contradiction. Hence we may assume $n\ge8$. 

We show that $\Delta(\m B(1))\le 2$. Indeed, if $\Delta(\m B(1))\ge 4$, then  there is $x\in [2,n]$ and four different elements $x_1,x_2,x_3, x_4\in [2,n]\setminus \{x\}$ such that all $\{1,x,x_1\}, \{1,x,x_2\}, \{1,x,x_3\}, \{1,x,x_4\}$ belong to $\m B$. This forces every set in $\m A(\bar1)$ to contain $x$. This implies that $|\m A(x)|\ge|\m A(\bar1)|\ge 2n-6$, contradicting $\Delta(\m A)\le n-2$. If $\Delta(\m B(1))= 3$, then w.l.o.g., we may assume $\{1,x,x_1\}, \{1,x,x_2\}, \{1,x,x_3\}$ belong to $\m B$, this implies that $\m A(\bar1\bar x)\subset \{\{x_1,x_2,x_3\}\}$, and hence $|\m A(x)|\ge|\m A(\bar1)|-1\ge 2n-7$, contradicting $\Delta(\m A)\le n-2$.

Note that $\tau(\m B(1))\le 3$.
If $\tau(\m B(1))\le2$, then using $\Delta(\m B(1))\le 2$, we obtain $\Delta(\m B)\le 4$, and hence $|\m B|=\frac{1}{3}\sum_{i\in [n]}d_{\m B}(i)\le \frac{4n}{3}<3n-8$ (by $n\ge 8$), a contradiction. 
It remains to consider $\tau(\m B(1))=3$. Let $\nu(\m B(1))$ denote the matching
number of $\m B(1)$. Then $\nu(\m B(1))\in\{2,3\}.$
If $\nu(\m B(1))=3$, then $|\m A(\bar1)|\le 2^3=8$, and hence $|\m A|\le n-2+8<3n-8$, a contradiction. 
If $\nu(\m B(1))=2$, fix a maximum matching: $E_1,E_2\in \m B(1)$. Given a
$3$-element vertex cover, choose one of its vertices from each of
$E_1,E_2$. There are at most four possible choices for this pair.
Since $\tau(\m B(1))=3$, the chosen pair is not a vertex cover. The third
vertex must therefore cover all sets missed by the pair, and there
are at most two possibilities for this element. Thus $|\m A(\bar1)|\le 8$, again $|\m A|\le n-2+8<3n-8$, a contradiction. 
\end{proof}

\begin{lemma}\label{end2}
Let $n\ge 7$, $\m A, \m B \subset {[n]\choose 3}$ be non-trivial cross-intersecting families. 
If $|\m B(1)|\ge n-1$, then $|\m B(1)|+|\m A(\bar1)|\le 3n-8$, and equality holds only in one of the following three cases.
\begin{itemize}
\item[\rm (i)]
There exists $3$-set $S\subset [2,n]$ such that
$\m A(\bar1)=\{S\}$ and 
$\m B(1)=\left\{E\in\binom{[2,n]}{2}:E\cap S\ne\emptyset\right\}$.
\item[\rm (ii)]
There exists $2$-set $P\subset [2,n]$ such that $\m B(1)=\left\{E\in\binom{[2,n]}{2}:E\cap P\ne\emptyset\right\}$ and
$\m A(\bar1)=\{E\in \binom{[2,n]}{3}:P\subset E\}$.
\item[\rm (iii)]
There exist $x\in [2,n]$ and $P\in\binom{[2,n]\setminus\{x\}}2$ such that
$\m A(\bar1)=\{E\in \binom{[2,n]}{3}:x\in E, P\cap E\ne\emptyset\}$ and 
$\m B(1)=\left\{E\in\binom{[2,n]}{2}:x\in E\right\}\cup\{P\}$.
\end{itemize}
Moreover, $|\m A(\bar1)|\le 2n-7$. 
\end{lemma}

\begin{proof}
Note that
every set of $\m A(\bar 1)$ is a $3$-element vertex cover of
$\m B(1)$. Thus, $\tau(\m B(1))\le 3$.
Since $|\m B(1)|>n-2$, $\tau(\m B(1))\ge 2$. Together with $\tau(\m B(1))\le3$, we have $\tau(\m B(1))\in\{2,3\}.$ 

{\bf Case 1: $\tau(\m B(1))=2$. }
Let $\{u,v\}$ be a minimum vertex
cover. Put 
\[
a:=|\m B[\{1,u\}\overline{\{v\}}]|,
\quad
b:=|\m B[\{1,v\}\overline{\{u\}}]|,
\quad
\eta:=|\m A(\overline{\{1,u,v\}})|,
\quad
\varepsilon:=
\begin{cases}
1,& \text{ if } \{1,u,v\}\in \m B,\\
0,& \text{otherwise}.
\end{cases}
\]
Then $|\m B(1)|=a+b+\varepsilon,$ $a,b\ge1$, $|\m A(\{u,v\}\overline{\{1\}})|\le n-3$, $a,b \le n-3$,
and if $\eta=1$, then $\varepsilon=0$.

First, consider the case $a,b\ge3$. If $\varepsilon=1$, then every set in $\m A(\bar1)$ contains $\{u,v\}$, i.e. $|\m A(\bar1)|\le {n-3}$, so $|\m A(\bar1)|+|\m B(1)|\le {n-3}+a+b+1\le 3(n-3)+1=3n-8$.
If $\varepsilon=0$, then either $\eta=1$, in which case $a=b=3$, $|\m B(1)|=6$, and $|\m A(\bar1)|\le n-2$, and hence $|\m B(1)|+|\m A(\bar1)|<3n-8$; or $\eta=0$, in which case 
$|\m A(\bar1)|\le n-3$ and $|\m B(1)|\le 2(n-3)$, and hence, 
$|\m B(1)|+|\m A(\bar1)|\le 3(n-3)<3n-8$. Thus, if $a,b\ge3$, then $|\m B(1)|+|\m A(\bar1)|\le 3n-8$ with equality only if $a=b=n-3$ and $\varepsilon=1$, which is case \rm(ii). Moreover, $|\m A(\bar1)|\le n-2<2n-7$. 

Next consider the case $a=2$ or $b=2$. If, say, $b=2$, then $a\ge n-4$, $\eta=0$, $|\m A(\bar1)|\le n-3+1=n-2$, and $|\m B(1)|\le 1+2+n-3=n$, thus,
$|\m B(1)|+|\m A(\bar1)|\le n-2+n<3n-8$. In this case, $|\m A(\bar1)|\le n-2<2n-7$. 

Finally consider the case $a=1$ or $b=1$. If, say, $b=1$, then $|\m B(1)|\ge n-1$ and  $a\le n-3$ forces $a= n-3$ and $\varepsilon=1$. 
Consequently, $|\m B(1)|= n-1$ and $|\m A(\bar1)|\le 2n-7$,
and case \rm(iii) follows. 

{\bf Case 2: $\tau(\m B(1))=3$.}
Let $S=\{x_1,x_2,x_3\}$ 
be a minimum vertex cover. 
For each $i\in[3]$, let $a_i:=|\m B(\{1,x_i\}\overline{S\setminus \{x_i\}})|$, and let $q:=|\{B\in \m B(1): |B\cap S|=2\}|$.
Since $S$ is a minimum cover, $a_i\ge1$ for every $i$, and $|\m B(1)|=a_1+a_2+a_3+q.$
Besides $S$ itself, a three-element cover containing exactly two members of $S$ is obtained by omitting $x_i$ precisely when $a_i=1$. Let $s:=|\{i:a_i=1\}|$. Then $s=|\{A\in \m A(\bar1): |A\cap S|=2\}|$.
Let $t:=|\{A\in \m A(\bar1): |A\cap S|=1\}|$.
Then $t+q\le 3$.
Let $z:=|\{A\in \m A(\bar1): A\cap S=\emptyset\}|$. Then $z\in\{0,1\}$.
Thus $|\m A(\bar1)|=1+s+t+z$.
Let $r:=|[2,n]\setminus S|$. Note that $s\le\sum_{i=1}^3(r-a_i).$ 
Then $|\m A(\bar1)|=1+s+t+z \le 1+ \sum_{i=1}^3(r-a_i)+3-q.$
Indeed, this is immediate when $z=0$. If $z=1$, then $q=0$; the only additional exceptional possibility is $t=3$, in which case the three sets $a_1=a_2=a_3=1$, and the displayed inequality is again strict.
Therefore 
\[
|\m A(\bar1)|+|\m B(1)|\le 1+ \sum_{i=1}^3(r-a_i)+3-q+a_1+a_2+a_3+q=3r+4=3n-8.
\]
If equality holds, then $s=\sum_{i=1}^3(r-a_i).$ 
Since $r=n-4\ge3$, this forces $a_1=a_2=a_3=n-4.$
It then follows that $t=z=0$ and $q=3$. Thus $\m B(1)$ consists of all
edges meeting $S$, and $S$ is its unique three-element vertex cover.
This is case \rm(i).
Note that $|\m A(\bar1)|\le 1+ \sum_{i=1}^3(r-a_i)+3-q$ and $|\m B(1)|=a_1+a_2+a_3+q$, since $|\m B(1)|\ge n-1$, $|\m A(\bar1)|\le 3n-8-|\m B(1)|\le 2n-7$.
This completes the proof.
\end{proof}

Now we are ready to prove Proposition \ref{end3}. 
\begin{proof}[Proof of Proposition \ref{end3}]
By Lemma \ref{end1}, $\Delta(\m A)\ge n-1$ or $\Delta(\m B)\ge n-1$. 
Assume $\Delta(\m B)\ge n-1$ and $\m B_{\Delta}=\m B(1)$. Then $|\m B(1)|\ge n-1$. By Lemma \ref{end2}, we have $|\m A(\bar1)|\le 2n-7$ and 
\begin{equation}\label{6-21-1}
|\m B(1)|+|\m A(\bar1)|\le 3n-8.
\end{equation}
Since $|\m A|= 3n-8$, $|\m A(1)|\ge n-1$, then applying Lemma \ref{end2} (when $\m A(1)$ and $\m B(1)$ are exchanged), we obtain 
\begin{equation}\label{6-21-2}
|\m A(1)|+|\m B(\bar1)|\le 3n-8.
\end{equation}
Since $|\m A|=|\m B|=3n-8$, $|\m A|+|\m B|=2(3n-8)$, which forces both equalities in (\ref{6-21-1}) and (\ref{6-21-2}) to hold. By Lemma \ref{end2}, two pairs $(\m A(1), \m B(\bar1))$ and $(\m A(\bar1), \m B(1))$ satisfy conditions (i) or (ii) or (iii). 
Combining with $|\m A|=|\m B|=3n-8$, a direct calculation shows that the two pairs satisfy the same condition, i.e., they both satisfy (i), both satisfy (ii), or both satisfy (iii); no mixed case can occur. If they both satisfy (i), then $(\m A, \m B)$ is of cross-HM type; if they both satisfy (ii), $(\m A, \m B)$ is of cross-triangle type; otherwise, $(\m A, \m B)$ is of cross-HM type or cross-triangle type.
\end{proof}

\end{document}